\documentclass[12pt]{amsart}
\usepackage{amssymb,amsthm}
\usepackage{hyperref} 
\usepackage{color}	

\usepackage[all]{xy}
\newtheorem{theorem}{Theorem}[section]
\newtheorem*{theorem*}{Theorem}
\newtheorem{proposition}[theorem]{Proposition}
\newtheorem{coro}[theorem]{Corollary}
\newtheorem{lemma}[theorem]{Lemma}
\newtheorem{rem}[theorem]{Remark}
\newtheorem{definition}[theorem]{Definition}

\newtheorem{question}{Question}

\newtheorem*{assum-gamma}{Assumptions on $\Gamma$}
\newtheorem*{assum-L}{Assumptions on $L$}

\theoremstyle{plain}

\renewcommand{\epsilon}{\varepsilon}
\newcommand{\eps}{\epsilon}

\newcommand{\diam}{\textrm{diam}}
\newcommand{\dive}{\textrm{div}}

\newcommand\N{\mathbb{N}}
\newcommand\R{\mathbb{R}}

\newcommand{\calM}{\ensuremath{\mathcal{M}}}

\DeclareMathOperator{\supp}{supp}
\DeclareMathOperator{\loc}{loc}

\def\Xint#1{\mathchoice
   {\XXint\displaystyle\textstyle{#1}}%
   {\XXint\textstyle\scriptstyle{#1}}%
   {\XXint\scriptstyle\scriptscriptstyle{#1}}%
   {\XXint\scriptscriptstyle\scriptscriptstyle{#1}}%
   \!\int}
\def\XXint#1#2#3{{\setbox0=\hbox{$#1{#2#3}{\int}$}
     \vcenter{\hbox{$#2#3$}}\kern-.5\wd0}}

\def\aver#1{\Xint-_{#1}}

\makeatletter
\let\original@addcontentsline\addcontentsline
\newcommand{\dummy@addcontentsline}[3]{}
\newcommand{\DeactivateToc}{\let\addcontentsline\dummy@addcontentsline}
\newcommand{\ActivateToc}{\let\addcontentsline\original@addcontentsline}
\makeatother

\newcommand\restr[2]{{
  \left.\kern-\nulldelimiterspace 
  #1 
  \vphantom{\big|} 
  \right|_{#2} 
  }}

\numberwithin{theorem}{section}
\numberwithin{equation}{section}

\title[Riesz transforms through reverse H\"older and Poincar\'e inequalities]{Riesz transforms through reverse H\"older and Poincar\'e inequalities}

\author[Fr\'ed\'eric Bernicot, Dorothee Frey]{Fr\'ed\'eric Bernicot, Dorothee Frey}

\address{Fr\'ed\'eric Bernicot, CNRS - Universit\'e de Nantes, Laboratoire Jean Leray, 2 rue de la Houssini\`ere, 44322 Nantes cedex 3, France\\}
\email{frederic.bernicot@univ-nantes.fr}

\address{Dorothee Frey, CNRS - Universit\'e Paris-Sud, Laboratoire de Math\'ematiques, UMR 8628 du CNRS, 91405 {\sc Orsay}, France\\}
\email{dorothee.frey@univ-nantes.fr}

\thanks{This project is partly supported by ANR projects AFoMEN no. 2011-JS01-001-01 and HAB no. ANR-12-BS01-0013.}

\date{\today}

\begin{document}

\begin{abstract}  
 We study the boundedness of Riesz transforms in $L^p$ for $p>2$ on a doubling metric measure space endowed with a gradient operator and an injective, $\omega$-accretive operator $L$ satisfying Davies-Gaffney estimates. If $L$ is non-negative self-adjoint, we show that under a reverse H\"older inequality, the Riesz transform is always bounded on $L^p$ for $p$ in some interval $[2,2+\eps)$, and that $L^p$ gradient estimates for the semigroup imply boundedness of the Riesz transform in $L^q$ for $q \in [2,p)$. This improves results of \cite{ACDH} and \cite{AC}, where the stronger assumption of a Poincar\'e inequality and the assumption $e^{-tL}(1)=1$ were made. The Poincar\'e inequality assumption is also weakened in the setting of a sectorial operator $L$. 
 In the last section, we study elliptic perturbations of Riesz transforms.
\end{abstract}

\subjclass[2010]{58J35, 42B20}

\keywords{Poincar\'e inequalities, Riesz transform, Reverse H\"older inequality, harmonic functions}

\maketitle

\section{Introduction}

The $L^p$ boundedness of Riesz transforms on manifolds has been widely studied in recent years. 
The Riesz transform has to be thought of as  ``one side" (or one half) of the commutator between two first order operators: the gradient, coming from the metric structure of the underlying space, and the square root of a second order operator under consideration (e.g. the Laplace-Beltrami operator on a Riemannian manifold). Thus, Riesz transform bounds allow to compare the two corresponding first order homogeneous Sobolev spaces. The aim of this article is to give new sufficient criteria for the boundedness of Riesz transforms in $L^p$ for $p>2$, and to study its stability under elliptic perturbations.\\

Before describing our results in detail, let us introduce the setting and some notation.

\subsection{Setting}
\label{sec11}

Let $(X,d)$ be a locally compact separable metric space, equipped with a Borel measure $\mu$, finite on compact sets and strictly positive on any non-empty open set. For $\Omega$ a measurable subset of $X$, we shall  denote $\mu\left(\Omega\right)$ by $\left|\Omega\right|$.  For all $x \in X$ and all $r>0$, denote by $B(x,r)$ the open ball for the metric $d$ with centre $x$ and radius $r$, and  by $V(x,r)$ its measure $|B(x,r)|$.  For a ball $B$ of radius $r$ and a real $\lambda>0$, denote by $\lambda B$   the ball concentric  with $B$ and with radius $\lambda r$. We shall sometimes denote by $r(B)$ the radius of a ball $B$. We will use $u\lesssim v$ to say that there exists a constant $C$ (independent of the important parameters) such that $u\leq Cv$, and $u\simeq v$ to say that $u\lesssim v$ and $v\lesssim u$. Moreover, for $\Omega\subset X$ a subset of finite and non-vanishing measure and $f\in L^1_{loc}(X,\mu)$, $\aver{\Omega} f \, d\mu=\frac{1}{|\Omega|} \int f \, d\mu$ denotes the average of $f$ on $\Omega$. 

From now on, we assume that $(X,d,\mu)$ is a doubling metric measure space, which means that the measure $\mu$ satisfies the doubling property, that is
  \begin{equation}\label{d}\tag{$V\!D$}
     V(x,2r)\lesssim  V(x,r),\quad \forall~x \in X,~r > 0.
    \end{equation}
As a consequence, there exists  $\nu>0$  such that
     \begin{equation*}\label{dnu}\tag{$V\!D_\nu$}
      V(x,r)\lesssim \left(\frac{r}{s}\right)^{\nu} V(x,s),\quad \forall~r \ge s>0,~ x \in X,
    \end{equation*}
    
We abstractly define a gradient operator $\Gamma$, in terms of which we   express first order regularity on the metric space.

\begin{assum-gamma}
Assume that there exists a sublinear operator $\Gamma$, with dense domain ${\mathcal F} \subset L^2(X,\mu)$. Assume that $\Gamma$ is a local operator, which means that for every $f \in {\mathcal F}$,
$$ \supp \Gamma f \subseteq \supp f.$$ 
Moreover, assume that $\Gamma$ satisfies a \emph{relative Faber-Krahn inequality}, that is for every ball $B$ of radius $r\leq \delta \,\diam(X)$ with some $\delta<1$, and all  $f\in\mathcal{F}$ supported in $B$,
\begin{equation}  \label{FKR}
	\left(\int_{B} |f|^2\,d\mu\right)^{1/2}
	\lesssim r \left(\int_{B} |\Gamma f|^2\,d\mu\right)^{1/2}.
\end{equation}
\end{assum-gamma}

We then consider an unbounded operator $L$ on $L^2(X,\mu)$ under the following assumptions.

\begin{assum-L}
Assume that  $L$ is an injective, $\omega$-accretive operator with domain $\mathcal D\subset {\mathcal F}$ in $L^2(X,\mu)$, where $0 \leq \omega < \pi/2$.
Assume that for all $f\in {\mathcal D}$,
 \begin{equation}\tag{$R_2$}
\| \Gamma f \|_2 \lesssim \|L^{1/2} f\|_2.
 \label{E2}
\end{equation}
Assume that $L$ satisfies $L^2$ Davies-Gaffney estimates, which means that 
for every $r>0$ and all balls $B_1$,$B_2$ of radius $r$
\begin{equation} \tag{$DG$}
\| e^{-r^2L} \|_{L^2(B_1) \to L^2(B_2)} + \| r\Gamma e^{-r^2L} \|_{L^2(B_1) \to L^2(B_2)} \lesssim e^{-c \frac{d^2(B_1,B_2)}{r^2}}.
\label{eq:DG}
\end{equation}
\end{assum-L}

By our assumptions, $L$ is an injective, maximal accretive operator on $L^2(X,\mu)$, and therefore has a bounded $H^\infty$ functional calculus on $L^2(X,\mu)$. The assumption $\omega<\frac{\pi}{2}$ implies that $-L$ is the generator of an analytic semigroup in $L^2(X,\mu)$. 
See \cite{ADM,Kato} for definitions and further considerations, and Section \ref{sec2} for examples.\\

We will assume the above throughout the paper. We abbreviate the setting with $(X,\mu,\Gamma,L)$.

\subsection{Notation}

For a (sub)linear operator $T$ and two exponents $1\leq p\leq q \leq \infty$, we say that $T$ satisfies \emph{$L^p$-$L^q$ off-diagonal estimates at scale $r>0$} if there exist implicit constants such that for all balls $B_1$,$B_2$ of radius $r$ and every function $f\in L^p(X,\mu)$ supported in $B_1$, we have
$$ \left( \aver{B_2} |Tf|^q \, d\mu \right)^{1/q} \lesssim \exp\left( -c \frac{d^2(B_1,B_2)}{r^2}\right) \left( \aver{B_1} |f|^p \, d\mu \right)^{1/p}.$$
Denote 
\begin{align*}
 p_L &:= \inf \left\{ p\in(1,2];\ \textrm{for all $t>0$, $e^{-tL}$ satisfies $L^p$-$L^2$ off-diagonal estimates} \right\},\\
 p^L &:= \sup \left\{ p\in[2,\infty);\ \textrm{for all $t>0$, $e^{-tL}$ satisfies $L^2$-$L^p$ off-diagonal estimates} \right\}.
\end{align*}
We assume that $p_L \neq 2$ and $p^L \neq 2$.\\
By composing off-diagonal estimates, it follows that for every $p,q\in(p_L,p^L)$ with $p\leq q$, the semigroup $e^{-tL}$ satisfies $L^p$-$L^{q}$ off-diagonal estimates. We deduce that for every $p\in(p_L,p^L)$,
$$ \sup_{t>0} \|e^{-tL}\|_{p \to p} < \infty,$$
and $(e^{-tL})_{t>0}$ is   bounded analytic on $L^p(X,\mu)$, see \cite[Corollary 1.5]{BK2}. In particular, it means that  $(tLe^{-tL})_{t>0}$ is bounded on $L^p(X,\mu)$ uniformly in $t>0$.

For $p\in(1,\infty)$, one says that $(R_p)$ holds if the Riesz transform ${\mathcal R}:=\Gamma L^{-1/2}$ is bounded on $L^p(X,\mu)$, which means
\begin{equation} \label{rp}
 \|\Gamma f \|_{p} \lesssim \| L^{1/2}f \|_{p},  \quad \forall\,f\in {\mathcal D}, \tag{$R_p$}
\end{equation}
and that estimates $(G_p)$ on the gradient of the semigroup holds if
\begin{equation} \label{Gp}
\sup_{t>0} \|\sqrt{t}\Gamma e^{-tL} \|_{p\to p} <+\infty \tag{$G_p$}.
\end{equation}
We refer the reader to \cite{ACDH} for the introduction of this notion, and where the investigation of the link between $(G_p)$ and $(R_p)$ has started. In the following, whenever we talk about \eqref{Gp}, we shall implicitly assume $p \in (p_L,p^L)$. \\

Let us now formulate scale-invariant Poincar\'e inequalities on $L^p(X,\mu)$, which may or may not be true. More precisely, for $p\in[1,+\infty)$, one says that  $(P_p)$ holds if
\begin{equation}\tag{$P_p$}
 \left( \aver{B} | f - \aver{B} f d\mu |^p d\mu \right)^{1/p} \lesssim r \left(\aver{B} |\Gamma f|^p \, d\mu \right)^{1/p}, \qquad \forall\, f\in {\mathcal F}, \label{Pp}
\end{equation}
where $B$ ranges over balls in $X$ of radius $r$.
Recall that $(P_p)$ is weaker and weaker as $p$ increases, that is $(P_p)$ implies $(P_q)$ for $q>p$, see for instance \cite{HaKo}.  The version for $p=\infty$ is trivial in the Riemannian setting.

\subsection{Main results and state of the art}

As started in \cite{ACDH}, it is natural to study the connection between gradient estimates $(G_p)$ and the boundedness of the Riesz transform $(R_p)$. It is easy to see that for every $p\in(p_L,p^L)$, $(R_p)$ implies $(G_p)$ by analyticity of the semigroup in $L^p(X,\mu)$.
For $p=2$, we have \eqref{E2} by assumption, and therefore also $(G_2)$. 
Following \cite{CD,BK1,HM}, it is known that $(G_p)$ and $(R_p)$ also hold for every $p\in(p_L,2)$.

But the following question still remains open:

\begin{question} \label{Q2} For $p\in(2,p^L)$, does $(G_p)$ imply $(R_p)$ or at least $(R_{q})$ for every $q\in(2,p)$?
\end{question}

This question is still open in the general framework we are considering here, and no counter-example is known. We refer the reader to the last paragraph of Subsection \ref{subsec:ex} where we describe two situations where $(G_p)$ is known but not $(R_p)$. So it is a very deep question to understand which extra property (or maybe none) is sufficient / necessary to deduce $(R_p)$ from $(G_p)$.

 In \cite{ACDH}, a positive answer has been obtained under the additional assumption of the Poincar\'e inequality $(P_2)$ and the conservation property. Here, we say that $L$ satisfies the \emph{conservation property}, if $e^{-tL}(1)=1$ for all $t>0$ (\footnote{We remark that the assumption \eqref{eq:DG} allows to define $e^{-tL}$ as an operator from $L^\infty(X,\mu)$ to $L^2_{\loc}(X,\mu)$.}), and that $\Gamma$ satisfies the \emph{conservation property}, if for every $\phi \in {\mathcal F}$, one has $\Gamma \phi=0$ on every ball $B$ such that $\phi$ is constant on $B$.
  With our notation, the result in \cite{ACDH} states as

\begin{theorem*}[\cite{ACDH}]
Let $(X,\mu,\Gamma,L)$ be as in Section \ref{sec11}. Assume $(P_2)$, $p^L=\infty$, and assume that $L$ and $\Gamma$ satisfy the conservation property. Suppose $p \in (2,\infty)$. If $(G_p)$ holds, then $(R_{q})$ holds for every $q\in(2,p)$.
\end{theorem*}

 However, in order to have $(R_p)$ for some $p>2$, it is known that neither $(P_2)$ is a necessary assumption, as the example of \cite{CCH} for the two copies of $\R^n$ glued smoothly along their unit circles shows, nor is the conservation property, as the example of some Schr\"odinger operator shows.

In the present work, we shall push the argument of \cite{ACDH}  further by weakening these two assumptions $(P_2)$ and conservation property. We consider two situations, first the more specific situation of a non-negative, self-adjoint operator, and then, in a second part, the situation of a sectorial operator as described above. In both cases, we are able to relax on the assumptions imposed in \cite{ACDH}. 

 Let us first discuss the case of non-negative, self-adjoint operators $L$. Here, we can in particular make use of the fact that the assumption \eqref{eq:DG} is equivalent to the finite propagation speed  property of $\sqrt{L}$. This fact was observed in \cite{Sikora}, and it was used there in order to study the boundedness of the Riesz transform for $p<2$, as well as for $p>2$ on $1$-forms. We then introduce a property $(RH_p)$ (see Definition \ref{def:RHp}), which describes a reverse local H\"older inequality for the gradient of harmonic functions. This property already appeared in \cite{Gia0} for the solutions of elliptic PDEs. In the context of Riesz transform bounds, it was already used in various results for Schr\"odinger operators \cite{Shen1,ABA,Shen0}, and in \cite{AC} for the Laplace-Beltrami operator on a Riemannian manifold. In Subsection \ref{subsec:r}, we prove the following.

\begin{theorem} \label{Thm:rh} Let $(X,\mu,\Gamma,L)$ be as in Section \ref{sec11}. Assume that $L$ is a non-negative, self-adjoint operator, with $p^L>\nu$ or with the additional assumption \eqref{eq:RR2}.  Suppose $p\in(2,p^L)$. If $(G_p)$ and $(RH_p)$ hold, then $(R_{q})$ holds for every $q\in(2,p)$.
\end{theorem}

Following \cite[Theorem 6.3]{BCF1}, we know that if $L$ and $\Gamma$ satisfy the conservation property, the combination $(G_p)$ with $(P_p)$ for some $p>2$ implies $(P_2)$, and so implies $(R_q)$ for every $q\in(2,p)$ by \cite{ACDH}.

Since we will also check that $(G_p)$ with $(P_p)$ (and in particular $(P_2)$) implies $(RH_p)$ (see Proposition \ref{prop:pprp}), this new result improves \cite{ACDH}. Moreover, we remove the assumption of the conservation property. This allows to apply the result to e.g. Schr\"odinger operators.
However, Question A still remains open in its generality. since for the two glued copies of $\R^n$ with its Riemannian gradient and the Laplace operator, we have $(G_p)$ and $(R_p)$ for $p\in (1,n)$, but a simple argument shows that $(RH_p)$ is not satisfied for $p>2$ (we leave it to the reader to check this).

Theorem \ref{Thm:rh} will be proved by an extensive use of harmonic functions and a fundamental lemma, which allows to locally ``approximate" (in some sense) a function by a harmonic function, see Lemma \ref{lem:LM}.\\

With similar methods, we can then also give a partial answer to the following question. 

\begin{question} \label{Q3} 
Suppose $p\in[2,p^L)$. Under which condition does $(G_p)$ imply $(R_{p+\eps})$ for some $\eps>0$?
\end{question} 

It is clear that in general, $(G_p)$ cannot imply $(R_{p+\eps})$ without any additional assumption. A counterexample is again the example of the two glued copies of $\R^n$. For $n=2$, it is known that for every $p>2$, $(R_p)$ does not hold, but $(R_2)$ and $(G_2)$ hold trivially. In Subsection \ref{subsec:ppgp}, we prove the following sufficient criterion.

\begin{theorem} \label{thm:imp} Let $(X,\mu,\Gamma,L)$ be as in Section \ref{sec11} with both $\Gamma$ and $L$ satisfying the conservation property. Assume that $L$ is a non-negative, self-adjoint operator. If $(G_p)$ and $(P_p)$ hold for some $p\geq 2$, then there exists $\eps>0$ such that $(R_{p+\eps})$ and $(G_{p+\eps})$ hold.
\end{theorem}

Then in the second part (Section \ref{sec:riesz}), we will focus on the situation where the operator $L$ is only sectorial and not necessarily self-adjoint. We show

\begin{theorem} \label{thm:na1}
Let $(X,\mu,\Gamma,L)$ be as in Section \ref{sec11}, with the conservation property for $L$ and $\Gamma$. Suppose $p\in(2,p^L)$. If $(G_p)$ and $(P_p)$ hold, then $(R_{q})$ holds for every $q\in(2,p)$.
\end{theorem}

In such a situation, it is unknown if $(G_p)$ together with $(P_p)$ implies $(P_2)$. The proof of this statement in \cite[Theorem 6.3]{BCF1} relies on the self-adjointness of the operator, which is used through the finite  propagation speed property to deduce a perfectly localised Caccioppoli inequality. So as far as we know, the assumptions are actually weaker than those in \cite{ACDH}.\\

In the last section, we finally show that reverse H\"older inequalities in $L^p$ are stable under small elliptic perturbations for $p$ close to $2$.

\begin{theorem} \label{thm:per} Let $(M,\mu)$ be a doubling Riemannian manifold with $\nabla$ the Riemannian gradient and $\Delta$ the Laplace-Beltrami operator. Assume that the heat kernel of $(e^{t\Delta})_{t>0}$ satisfies pointwise Gaussian estimates. If the reverse H\"older property $(RH_{q})$ holds for $-\Delta$ and for some $q>2$, then there exists $\varepsilon>0$ such that for every $p\in(2,2+\varepsilon)$ and every map $A$ with $ \| A-\textrm{Id}\|_\infty \leq \varepsilon$ and $L_A=-\nabla^* A \nabla $ self-adjoint, the properties $(RH_{p})$ and $(R_{p})$ are satisfied with respect to the operator $L_A$.
\end{theorem}

\section{Examples of Applications and Preliminaries}
\label{sec2}

\subsection{Examples and Applications} \label{subsec:ex}

\begin{itemize}

\item {\bf Dirichlet forms and subdomains.} 

 Let $(M,d,\nu)$ be a complete space of homogeneous type as above. Consider a self-adjoint operator ${\mathcal L}$ on $L^2(M,\nu)$ and consider $\mathcal{E}$ the quadratic form associated with ${\mathcal L}$, that is
$${\mathcal E}(f,g)=\int_M f {\mathcal L}g\,d\nu.$$ 
If ${\mathcal E}$ is a strongly local and regular Dirichlet form (see \cite{FOT, GSC} for precise definitions) with a carr\'e du champ structure, then with $\Gamma$ being equal to this carr\'e du champ operator,  \eqref{E2} and \eqref{eq:DG} hold. In particular, this is the case if $(M,d,\nu)$ is a Riemannian manifold with $L$ its non-negative Laplace-Beltrami operator and $\Gamma$ the length of the Riemannian gradient.

\medskip

We consider the two following situations:

 \begin{enumerate}
 
 \item Consider the global situation $(X,d,\mu,L)=(M,d,\nu,{\mathcal L})$. In such a case, the typical upper estimate for the kernel of the semigroup is
  \begin{equation}\tag{$DU\!E$}
 p_{t}(x,y)\lesssim
\frac{1}{\sqrt{V(x,\sqrt{t})V(y,\sqrt{t})}}, \quad \forall~t>0,\,\mbox{a.e. }x,y\in
 X.\label{due}
\end{equation}
Under  \eqref{d}, \eqref{due} self-improves into a Gaussian upper estimate 
\begin{equation}\tag{$U\!E$}
p_{t}(x,y)\lesssim
\frac{1}{V(x,\sqrt{t})}\exp
\left(-c\frac{d^{2}(x,y)}{t}\right), \quad \forall~t>0,\, \mbox{a.e. }x,y\in
 X.\label{UE}
\end{equation}
See \cite[Theorem 1.1]{Gr1} for the Riemannian case,   \cite[Section 4.2]{CS} for a metric measure space setting. This yields $(p_{\mathcal L},p^{\mathcal L})=(1,\infty)$, and the Faber-Krahn inequality \eqref{FKR} holds. 

In such a situation, the operator $L$ satisfies the conservation property \cite{GSC}, and the space $(X,d,\mu)$ may or may not satisfies a Poincar\'e inequality $(P_p)$.

\item Local situation with Dirichlet boundary conditions. Given an open subset $U\subset M$, we can also look at the heat semigroup in $U$ with Dirichlet boundary conditions. So consider ${\mathcal L}_U$ the operator with domain 
$${\mathcal D}_U:=\left\{f\in {\mathcal D}({\mathcal L}), \ \textrm{supp}(f) \subset U\right\}.$$
Then it is known that ${\mathcal L}_U$ generates a semigroup in $L^2(U,\nu)$. 

However, it is important to emphasise that due to the Dirichlet boundary conditions, the conservation property does not hold for ${\mathcal L}_U$!

If the domain $U$ is assumed to be Lipschitz, then under \eqref{due} for the underlying space, it is known that the semigroup associated with ${\mathcal L}_U$ has also a heat kernel with Gaussian pointwise bounds and so $(p_{{\mathcal L}_U},p^{{\mathcal L}_U})=(1,\infty)$. Moreover the (eventual) Poincar\'e inequality $(P_2)$  on the ambiant space $(M,d,\nu)$ can be used to study the balls intersecting the boundary $\partial U$ (as done in e.g. \cite{Gia0,Shen1}). In particular, it is proved that if $(M,d,\nu)$ satisfies  $(P_2)$ and $U$ is Lipschitz, then for some $p>2$ a reverse H\"older property $(RH_p)$ holds for $L={\mathcal L}_U$, as well as $(G_p)$ and $(R_p)$.

Assume that $(M,d,\nu,{\mathcal L})$ is doubling and satisfies a Poincar\'e inequality $(P_2)$ with \eqref{due}. Consider $U\subset M$ an unbounded inner uniform open subset (see \cite{GSC} for a precise definition), which includes the case of Lipschitz domains. Then \cite[Theorem 3.12]{GSC} shows that $U$ equipped with its inner structure is also doubling. In order to get around the fact that ${\mathcal L}_U$ is not conservative, we can use the $h$-transform (as described in \cite{GSC}): let $h>0$ be a reduite of $U$ (which is a positive solution on $U$ of ${\mathcal L} h=0$ with Dirichlet conditions) and consider the operator
$$ {\mathcal L}_U^h f  = h^{-1} {\mathcal L_U} (h f)$$
with domain
$${\mathcal D}[{\mathcal L}_U^h]:=\{f\in L^2(U,h^2d\mu),\ hf \in {\mathcal D}_U \}.$$
Then ${\mathcal L}_U^h$ is self-adjoint with respect to the measure $h^2 d\mu$ and satisfies the conservation property.
According to \cite{GSC}, this structure defines a strongly local and regular Dirichlet form with a Poincar\'e inequality $(P_2)$ with respect to its carr\'e du champ $\Gamma:= |\nabla|$. Moreover the semigroup associated with ${\mathcal L}_U^h$ has also a heat kernel with Gaussian pointwise bounds (relatively to the measure $h^2 d\mu$) so that $(p_{{\mathcal L}_U^h},p^{{\mathcal L}_U^h})=(1,\infty)$ (see \cite{GSC}). So by our Theorem \ref{thm:main2}, there exists $\epsilon>0$ such that $(R_p)$ holds for $p\in[2,2+\epsilon)$ which means
$$ \left( \int_U |\nabla f|^p h^2\, d\mu \right)^{1/p} \lesssim \left( \int_U |h^{-1} {\mathcal L_U}^{1/2} (h f)|^p h^2\, d\mu \right)^{1/p},$$
which corresponds to the $L^p(U, h^2d\mu)$-boundedness of $\nabla h^{-1} {\mathcal L_U}^{1/2} (h \cdot)$. Moreover, we also have $(RH_p)$ for $p>2$ close to $2$.

\end{enumerate}

\medskip

\item {\bf Schr\"odinger operators.}

If $(M,d,\mu)$ is a doubling Riemannian manifold with $\nabla$ the Riemannian gradient, we may consider the Schr\"odinger operator $L:=-\nabla^* \nabla + V$ associated with a potential $V$. Let us focus on the case $V\geq 0$, otherwise the situation is more difficult and we refer the reader to \cite{Ass,AssO} (and references therein) for some works giving assumptions on $V$ to guarantee the boundedness of the Riesz transform. If $V$ is non-negative and belongs to some Muckenhoupt reverse H\"older space, then boundedness of the Riesz transform has been obtained in \cite{Shen1,ABA}.

For $V\geq 0$, we may consider $\Gamma$ given by
$$ \Gamma(f) = \left(|\nabla f|^2 +  V|f|^2\right)^{1/2} \quad \textrm{or} \quad \Gamma(f) = |\nabla f|.$$
Moreover, if \eqref{due} holds on the heat semigroup generated by  $-\nabla^* \nabla$, then it also holds for $L$, hence $(p_L,p^L)=(1,\infty)$.

Because of the potential $V$, the operator $L$ does not satisfies the conservation property. If $(M,d,\mu)$ satisfies the Poincar\'e inequality $(P_2)$ with respect to $\nabla$, then it can be proved that $L$ satisfies a $(RH_p)$ property for some $p>2$ and with $\Gamma=|\nabla |$ (see Proposition \ref{prop:P2RH}).

As an application of Theorem \ref{thm:th}, we deduce the following result.

\begin{theorem} Assume that $(M,d,\mu)$ is a doubling Riemannian manifold satisfying \eqref{due} and $(P_2)$. Then for every potential $V\geq 0$ there exists $\epsilon>0$ such that for $L=-\nabla^* \nabla + V$ and $\Gamma=|\nabla |$, we have $(RH_q)$ as well as $(G_q)$ and $(R_q)$ 
for every $q\in(2,2+\epsilon)$.
\end{theorem}

It has to be compared with \cite{CMO} where a negative answer for $(G_p)$ (and so $(R_p)$) is given for $p>\nu$ if the operator $L$ has a positive ground state function.

\medskip

\item {\bf Second order divergence form operators.}

Consider $(M,d,\mu)$ a doubling Riemannian manifold (satisfying \eqref{dnu}), equipped with the Riemannian gradient $\nabla$ and its divergence operator $\dive = \nabla^*$. To $A = A(x)$ a complex matrix-valued function, defined on $M$ and satisfying the ellipticity (or accretivity) condition (see Section \ref{sec:appli} for details), we may define a second order divergence form operator
$$ L=L_A f :=- \dive (A\nabla f).$$
Then $L$ is sectorial and satisfies the conservation property but may not be self-adjoint. 

In this case, the reverse H\"older property $(RH_p)$ for some $p>2$ is implied by $(P_2)$, see \cite[Chapter 5, Thm 2.1]{Gia0}. We refer the reader to Section \ref{sec:appli} where a stability result (with respect to the map $A$) is shown for the reverse H\"older property.

\medskip

\item {\bf Examples where $(G_p)$ is known.}

We shall give here two examples of situations where $(G_p)$ is known for some $p>2$. However $(R_p)$ is still unknown, which motivates to ask Question A.

\begin{enumerate}
\item Consider $V$ a non-negative potential on ${\mathbb R}$ and define on ${\mathbb R}^2$ (equipped with its Euclidean structure), the operator for $x=(x_1,x_2) \in{\mathbb R}^2$
$$ Lf(x) = -\Delta f(x) +V(x_2)f(x).$$ 
Then $L$ generates a semigroup, which has a heat kernel satisfying Gaussian bounds. For $\Gamma=\partial_{x_1}$, since $V$ is very nice along the first coordinate, it is natural to ask for the boundedness of $\Gamma L^{-1/2}$. Such a boundedness is unknown and does not seem to be trivial, since $L^{-1/2}$ makes interact the action of $L$ on the two coordinates. However, it is surprising to observe that $(G_p)$ easily holds for every $p\in[1,\infty]$. Indeed, since $L_1:=-\partial_{x_1}^2$ and $L_2:=-\partial_{x_2}^2+V(x_2)$ commute, we then have $ e^{-tL} = e^{-tL_1} \otimes e^{-tL_2}$ such that
$$ \|\sqrt{t} \Gamma e^{-tL}\|_{p\to p} \leq \|\sqrt{t}\partial_{x_1} e^{-tL_1}\|_{p\to p} \|e^{-tL_2}\|_{p\to p}$$
which is uniformly (with respect to $t$) bounded. 

\item On ${\mathbb R}^n$, consider a second order operator 
$$ L = \sum_{|\alpha|\leq 2} c_\alpha(x) \partial^\alpha $$
given by bounded measurable complex coefficients $c_\alpha$, depending on $x$.
For some $q>2$, assume that the $L^q$ domain of $-L$ is the classical Sobolev space $D_q(-L)=W^{2,q}({\mathbb R}^n)$ and also is the generator of an analytic semigroup $(e^{-tL})_{t>0}$ in $L^q$.
Then \cite[Theorem 3.1]{Ku} shows that a local version of $(G_q)$ holds (with $\Gamma=\nabla$). In this context, $L^q$-boundedness of some local Riesz transforms is unknown. 
\end{enumerate}

\end{itemize}

\subsection{Preliminaries}

We start by recalling some technical results, which will be needed later.

We denote by $\calM$ the Hardy-Littlewood maximal operator, defined for $f \in L^1_{\loc}(X,\mu)$ and $x \in X$ by
\begin{equation}	\label{def:maxop}
	\calM f(x):= \sup_{B \ni x} \,\aver{B} |f| \,d\mu,
\end{equation} 
where the supremum is taken over all balls $B \subset X$ with $x \in B$. 
For $p \in [1,+\infty)$, we abbreviate by $\calM_p$ the operator defined by $\calM_p(f):=[\calM(|f|^p)]^{1/p}$, $f \in L^1_{\loc}(X,\mu)$. Note that 
$\calM$ is bounded in $L^q(X,\mu)$ for all $q \in (1,+\infty]$, cf. \cite[Chapitre III]{CW}. Consequently, $\calM_p$ is bounded in $L^q(X,\mu)$ for all $q \in (p,+\infty]$.

Let us recall Gehring's result \cite{Gehring} which describes a self-improvement of inequalities involving maximal functions. 

\begin{theorem} \label{thm:gehring} Let $(X,d,\mu)$ be a doubling space. Let $1\leq p<q<\infty$ be two exponents and $f \in  L^q_{\loc}(X,\mu)$ such that for some constant $C>0$, $1<\alpha<\beta<2$ and a ball $B$  we have for almost every $x\in 2B$
$$  \sup_{x\in Q \subset 2 Q \subset \alpha B} \ \left(\aver{Q} |f|^q \,d\mu\right)^{1/q} \leq C \sup_{x\in Q \subset \beta B} \ \left(\aver{Q} |f|^p \,d\mu\right)^{1/p}.$$
Then there exists $\epsilon:=\epsilon(C,p,q,\alpha,\beta,\nu)>0$ such that $f\in L^{q+\epsilon}(B)$ and
$$ \left(\aver{B} |f|^{q+\epsilon} \, d\mu\right)^{1/(q+\epsilon)} \lesssim \left(\aver{2B} |f|^{p} \, d\mu\right)^{1/p}.$$
\end{theorem}

The original result is due to Gehring \cite[Lemma 2]{Gehring}, with a modification in \cite[Appendix]{GM} which emphasises the restriction to sub-balls $Q$ instead of global maximal functions. There is a large amount of literature on the topic, with extensions to various settings (\cite{Gia0, Iw} and also \cite{MM}). The proof relies on a suitable Calder\'on-Zygmund decomposition.

Let us also recall the following equivalence between weak and strong Poincar\'e inequalities (combining \cite[Theorem 3.1]{HaKo} and \cite{KZ}).
\begin{theorem} \label{thm:wpoincare} Let $(X,\mu,\Gamma,L)$  be as in Section \ref{sec11}. Let $p\in(1,\infty)$ and let $f\in {\mathcal F}$. The (strong) $L^p$ Poincar\'e inequality $(P_p)$ for $f$ is equivalent to the weak version: there exists $\lambda>1$ such that
\begin{equation}\tag{$w$-$P_p$}
 \left( \aver{B} | f - \aver{B} f d\mu |^p d\mu \right)^{1/p} \lesssim r \left(\aver{\lambda B} |\Gamma f|^p d\mu \right)^{1/p}, 
 \label{w-Pp}
\end{equation}
where $B$ ranges over balls in $X$ of radius $r$.
\end{theorem}

We also recall the following (well-known) fact about Poincar\'e inequality (see for instance \cite[Theorem 5.1]{HaKo}, \cite[Theorem 2.7]{FPW} for similar statements).

\begin{lemma}\label{met} Let $(X,d,\mu)$  be  a  doubling metric  measure space with \eqref{dnu}.  Assume  that
$(P_p)$ holds for some $1\le p<+\infty$. Then if $q\in (p,+\infty)$ is such that  $\nu\left(\frac{1}{p}-\frac{1}{q}\right)\le 1$,
the following Sobolev-Poincar\'e inequality holds:
\begin{equation}\label{Ppq}\tag{$P_{p,q}$}
 \left(\aver{B_r}| f-\aver{B_r}f\,d\mu |^{q} \, d\mu \right)^{1/q}\lesssim r \left(\aver{B_r} |\Gamma f|^{p} \, d\mu \right)^{1/p},
  \end{equation}
for all $f\in\mathcal{F}$, $r>0$, and all balls $B_r$ with radius $r$.
\end{lemma}

We will use the following extrapolation result (\cite{ACDH}, \cite[Theorem 3.13]{AM}). 

\begin{proposition} \label{prop:extrapolation} 
Let $(X,d,\mu)$ be a doubling metric measure space.
Let $(A_t)_{t>0}$ be a family of linear operators, uniformly bounded in $L^2(X,\mu)$.
Let $T$ be a sublinear operator which is bounded on $L^2(X,\mu)$. Assume that for some $q \in (2,+\infty)$, every ball $B$ of radius $r>0$ and every $f\in L^2(X,\mu)$, we have
\begin{itemize}
\item $L^2$-$L^2$  estimates of $T(I-A_{r})$: 
\begin{equation} \label{eq:ass1}
 \left( \aver{B} |T(I-A_{r} )f |^{2} d\mu \right)^{1/2} \lesssim  \inf_{x\in B} \calM_2(f)(x);
\end{equation}
\item $L^2$-$L^q$  estimates of $T(A_{r})$:
\begin{equation} \label{eq:ass2} 
\left( \aver{B} |TA_{r} f |^{q} d\mu \right)^{1/q} \lesssim \inf_{x\in B} \, [\calM_2(Tf)(x) + \calM_2(f)(x)].
\end{equation}
\end{itemize}
Then, for every $p\in(2,q)$, $T$ is bounded on $L^p(X,\mu)$.
\end{proposition}

Let us state a technical result, which describes how a higher order of cancellation with respect to the operator $L$ allows us to gain integrability through off-diagonal estimates.

\begin{lemma} \label{lemma} 
Let $p \in [2,p^L)$, and let $K>\frac{\nu}{2p}$.
 Then for every ball $B$ of radius $r>0$ and every function $f\in L^2(X,\mu)$, we have
 \begin{align*}
 \left(\aver{B} |(r^2L)^{K}  e^{-r^2L} f|^{p} \, d\mu \right)^{1/p}  \lesssim \sup_{Q\supset B} \left(\aver{Q} |f|^2 \, d\mu \right)^{1/2},
\end{align*}
where the supremum is taken over all balls $Q$ containing $B$.
\end{lemma}

\begin{proof} 
We write $K=M-\rho$ with $M\geq 1$ an integer and $\rho\in(0,1]$. Let $B$ be a ball of radius $r$ and $f \in L^2(X,\mu)$. 
Using a Calder\'on reproducing formula, we can write 
\begin{equation}  (r^2L)^{M-\rho}  e^{-r^2L}f = c \int_0^\infty   \left(\frac{r^2}{s}\right)^{M-\rho}  (sL)^M e^{-(s+r^2)L}f  \, \frac{ds}{s}. \label{eq:decompspectral}
\end{equation}
Now recall by analyticity of the semigroup, $(tL)^M e^{-tL}$ also satisfies $L^2$-$L^{p}$ off-diagonal estimates.
This yields that for $s\leq r^2$, the operator $(r^2L)^M e^{-(s+r^2)L}$ satisfies $L^2$-$L^{p}$ off-diagonal  estimates at the scale $r$, so
\begin{align*}
 \left(\aver{B}| (r^2L)^M e^{-(s+r^2)L} f|^{ p} \, d\mu \right)^{1/ p} \lesssim \sup_{Q\supset B} \left(\aver{Q} |f|^2 \, d\mu \right)^{1/2}. 
\end{align*}
Consequently, 
\begin{align*}
& \left(\aver{B} | \int_0^{r^2}  \left(\frac{r^2}{s}\right)^{M-\rho}(sL)^M e^{-(s+r^2)L} f|^{ p} \, d\mu \right)^{1/  p} \\
&\qquad \lesssim \int_0^{r^2} \left(\frac{s}{r^2}\right)^{\rho}  \left(\aver{B} |(r^2L)^M e^{-(s+r^2)L} f|^{ p} \, d\mu \right)^{1/  p}  \, \frac{ds}{s}  \lesssim  \sup_{Q\supset B} \left(\aver{Q} |f|^2 \, d\mu \right)^{1/2}.
\end{align*}
For $s\geq r^2$ on the other hand, $(sL)^M e^{-(s+r^2)L}$ satisfies $L^2$-$L^{p}$ off-diagonal estimates at the scale $s\geq r^2$. Denoting by $\tilde B = \frac{\sqrt{s}}{r} B \supset B$ the dilated ball, we obtain in this case
\begin{align*}
 \left(\aver{B} |(sL)^M e^{-(s+r^2)L} f|^{ p} \, d\mu \right)^{1/ p}
 & \lesssim \left(\frac{\sqrt{s}}{r}\right)^{\frac{\nu}{ p}}  \left(\aver{\tilde B} |(sL)^M e^{-(s+r^2)L} f|^{ p} \, d\mu \right)^{1/ p}\\
& \lesssim \left(\frac{\sqrt{s}}{r}\right)^{\frac{\nu}{ p}} \sup_{Q\supset \tilde B} \left(\aver{Q} |f|^2 \, d\mu \right)^{1/2} \\
& \lesssim \left(\frac{\sqrt{s}}{r}\right)^{\frac{\nu}{ p}} \sup_{Q\supset B} \left(\aver{Q} |f|^2 \, d\mu \right)^ {1/2}.
\end{align*}
This gives for $M>\frac{\nu}{2p}+\rho$
$$ \int_{r^2}^{\infty} \left(\frac{r^2}{s}\right)^{M-\rho}  \left(\aver{B} |(sL)^M e^{-(s+r^2)L} f|^{ p} \, d\mu \right)^{1/ p}  \, \frac{ds}{s}  \lesssim   \sup_{Q\supset B} \left(\aver{Q} |f|^2 \, d\mu \right)^{1/2}.
$$
Putting the two parts together yields the conclusion.
\end{proof}

\subsection{Harmonic functions}

Let us first rigorously describe what we mean by harmonic functions. First note that the map 
$${\mathcal E}(f,g) :=  \left\langle Lf, g \right\rangle $$
 is a sesquilinear form defined on ${\mathcal D}(L)$. If $L$ is self-adjoint, then it can be extended on ${\mathcal D}(L^{1/2})$ and so in particular in ${\mathcal F}$.

\begin{definition} 
Let $u\in {\mathcal D}(L^{1/2})$ and $B$ be a ball. We say that $u$ is harmonic on $B$ if for every $\phi \in {\mathcal D}(L^{1/2})$ supported on $B$, one has
$$ {\mathcal E}(u, \phi)  =0.$$
\end{definition}

\begin{definition} 
\label{def:RHp} 
Suppose $p_0\in (2,\infty)$. We say that the reverse H\"older property \eqref{RHp} holds if for every ball $B$ of radius $r>0$ and every function $u\in {\mathcal F}$ which is harmonic on $2B$, one has
   \begin{equation}\tag{$RH_{p_0}$}
 \left(\aver{B} |\Gamma u |^{p_0} \, d\mu \right)^{1/p_0} \lesssim \left(\aver{2B} |\Gamma u |^{2} \, d\mu \right)^{1/2}.  \label{RHp}
\end{equation}
\end{definition}

\begin{rem} \label{rem:rh} By the self-improving of the RHS exponent of a reverse H\"older inequality (see \cite[Appendix B]{BCF1}), we deduce that a reverse H\"older property $(RH_{p_0})$ for some $p_0>2$ implies a $L^1$-$L^{p_0}$ reverse H\"older property: for every ball $B$ of radius $r>0$ and every function $u\in{\mathcal F}$ which is harmonic on $2B$, one has
$$ \left(\aver{B} |\Gamma u |^{p_0} \, d\mu \right)^{1/p_0} \lesssim \left(\aver{2B} |\Gamma u | \, d\mu \right).$$
\end{rem}

First of all, harmonic functions will play a crucial role, so let us detail the main tool which allows us to approximate a function with a harmonic function.

\begin{lemma} \label{lem:LM} Let $(X,\mu,\Gamma,L)$ be as in Section \ref{sec11} with $L$ non-negative and self-adjoint. 
Let $f\in {\mathcal D}$ and consider an open ball $B \subset X$. Then there exists $u\in {\mathcal F}$ such that $f-u\in {\mathcal F}$ is supported in the ball $B$ and $u$ is harmonic in $B$.
Moreover, we have
\begin{equation} \label{eq:grad}
\begin{split} 
 \left(\aver{B} |\Gamma(f-u)|^2 \, d\mu \right)^{1/2} & \lesssim r\left(\aver{B} |L f|^2 \, d\mu \right)^{1/2} \\
 \left(\aver{B} |\Gamma u|^2 \, d\mu \right)^{1/2} & \lesssim \left(\aver{B} |\Gamma f|^2 \, d\mu \right)^{1/2} + r\left(\aver{B} |L f|^2 \, d\mu \right)^{1/2}.  
 \end{split}
 \end{equation}
\end{lemma}

We follow the same scheme as in \cite[Lemma 4.6]{BCF1}, where a proof was given in the particular setting of a Dirichlet form. We explain here how can we extend it to the current more general framework.

\begin{proof}  Consider the space of functions 
$$ {\mathcal H} := \left\{ \phi \in {\mathcal D}(L^{1/2}) \subset L^2, \ \textrm{supp}(\phi) \subset B \right\} \subset {\mathcal F}.$$
Then, due to \eqref{FKR} and the local character of operator $\Gamma$, the application 
$$\phi \mapsto \|\phi\|_{{\mathcal H}}:= \|L^{1/2} \phi\|_{L^2} \gtrsim \| \Gamma \phi \|_{L^2(B)}$$
defines a norm on ${\mathcal H}$.  Consequently, ${\mathcal H}$ equipped with this norm is a Hilbert space, with the scalar product
$$ \langle \phi_1,\phi_2\rangle_{{\mathcal H}} := {\mathcal E}(\phi_1,\phi_2).$$ 
Since $f\in {\mathcal D}$, the linear form 
$$ \gamma:\phi \mapsto \langle Lf,\phi\rangle$$
is continuous on ${\mathcal H}$. Indeed, we have by \eqref{FKR}
$$ |\gamma(\phi)| =\left| \langle Lf,\phi\rangle\right| \lesssim r\left(\int_{B} |L f|^2 \, d\mu \right)^{1/2} \|\Gamma \phi\|_{L^2(B)}.$$
By the representation theorem of Riesz, there exists $v\in{\mathcal H}$ such that for every $\phi \in  {\mathcal H}$ 
$$  \int_{B} \sqrt{L} f \sqrt{L} \phi \, d\mu = {\mathcal E}(v,\phi)=\int \sqrt{L} v \sqrt{L} \phi \, d\mu.$$
We set $u:= f-v$ so that $v=f-u$ being in ${\mathcal H}$  is supported in  $B$. Moreover for every $\phi\in {\mathcal H}$, $\phi$ is supported in $B$ so the previous equality yields
$$ {\mathcal E}(u,\phi)= \int \sqrt{L} f \sqrt{L} \phi \, d\mu - \int \sqrt{L} v \sqrt{L} \phi \, d\mu =0.$$
So $u$ is harmonic in $B$.

Then observe that since $f-u$ is supported in $B$ and $u$ is harmonic on $B$ by using \eqref{E2}, we have
\begin{align*}
\aver{B} |\Gamma (f-u)|^2 \, d\mu & \simeq \frac{1}{|B|} \| \sqrt{L} (f-u)\|_2^2
\simeq \frac{1}{|B|} {\mathcal E}(f-u,f-u)  \\
&\simeq \frac{1}{|B|} {\mathcal E}(f,f-u).
\end{align*}
Since $f-u \in{\mathcal F}$ is supported on $B$, Property \eqref{FKR} yields
\begin{align*}
{\mathcal E}(f,f-u) \lesssim  r\left(\int_{B} |L f|^2 \, d\mu \right)^{1/2} \| \Gamma (f-u)\|_{L^2(B)},
\end{align*}
which gives
$$ \left(\aver{B} |\Gamma (f-u)|^2 \, d\mu\right)^{1/2}  \lesssim  r\left(\aver{B} |L f|^2 \, d\mu \right)^{1/2}$$
and so \eqref{eq:grad}.
\end{proof}

In the previous result, the last quantity in \eqref{eq:grad} may be removed if we make an extra assumption.

\begin{lemma} \label{lem:LM-ter} Let $(X,\mu,\Gamma,L)$  be as in Section \ref{sec11} with $L$ non-negative and self-adjoint. Assume in addition that for every $f\in {\mathcal D}$, every ball $B$ and every function $g\in {\mathcal F}$ supported in $B$
\begin{equation} \left|\langle Lf,g \rangle \right| \lesssim \| \Gamma f\|_{L^2(B)} \| \Gamma g\|_{L^2(B)}. \label{eq:RR2} \end{equation}
Then for every $f\in {\mathcal D}$ and every an open ball $B \subset X$, there exists $u\in {\mathcal F}$ such that $f-u\in {\mathcal F}$ is supported in the ball $B$ and $u$ is harmonic in $B$, that is, for every $\phi \in {\mathcal F}$ supported on $B$ 
$$ {\mathcal E}(u,\phi) =0.$$
Moreover, we have
\begin{equation}
\left(\aver{B} |\Gamma(f-u)|^2 \, d\mu \right)^{1/2} + \left(\aver{B} |\Gamma u|^2 \, d\mu \right)^{1/2} \lesssim \left(\aver{B} |\Gamma f|^2 \, d\mu \right)^{1/2}. \label{eq:grad-ter} \end{equation}
\end{lemma}

We let the reader check that the exact same proof as the one of Lemma \ref{lem:LM} holds, except that now we can directly control the linear map $\gamma$ with 
\begin{align*}
 |\gamma(\phi)| & =\left| \langle Lf,\phi\rangle\right| \lesssim \left(\int_{B} |\Gamma f|^2 \, d\mu \right)^{1/2} \left(\int_{B} |\Gamma \phi|^2 \, d\mu \right)^{1/2} \\
 & \lesssim \left(\int_{B} |\Gamma f|^2 \, d\mu \right)^{1/2} \|\phi\|_{{\mathcal H}}.
\end{align*}

\begin{rem} The new assumption \eqref{eq:RR2} can be thought of as a sharp localised reverse Riesz inequality in $L^2(M,\mu)$. It is in particular satisfied for second order elliptic operators $L$ of divergence form, with $\Gamma$ being the length of the gradient.
\end{rem}

\section{Boundedness of Riesz transforms of self-adjoint operators under reverse H\"older property}

In this section, $(X,\mu,\Gamma,L)$  will be as in Section \ref{sec11}, with the additional assumption that $L$ is self-adjoint in $L^2(X, \mu)$.

\subsection{About finite speed of propagation}

We first need some technical results on specific approximation operators having finite propagation speed. We recall that for a non-negative, self-adjoint operator $L$ on $L^2(X,\mu)$, Davies-Gaffney estimates \eqref{eq:DG} for the heat semigroup are equivalent to the fact that the solution of the corresponding wave equation satisfies the finite propagation speed property.  See e.g. \cite{Sikora} and \cite[Section 3]{CS}.
The self-adjointness allows us to work with a better functional calculus than just $H^\infty$ functional calculus, namely functional calculus based on the Fourier transform. As investigated in \cite{AMcM}, this calculus interacts nicely with the finite propagation speed, and yields sharper off-diagonal estimates for operators generated by the calculus.

\begin{lemma} \label{lemma0}
For every even function $\varphi \in \mathcal{S}(\R)$ with $\supp \hat{\varphi} \subseteq [-1,1]$ and every $r>0$, the operator $\varphi(r\sqrt{L})$ propagates at distance at most $r$. That is, for all $f \in L^2(X,\mu)$ with $\supp f \subseteq E \subseteq X$, one has
$$ \supp(\varphi(r\sqrt{L})f) \subseteq \{x \in X:\,d(\{x\},E) \leq r\}. $$
Moreover we have the following $L^2$ off-diagonal estimates: for
all Borel sets $E,F \subseteq X$, we have
\begin{equation} \label{eq:L2OD-fs}
\| \varphi(r\sqrt{L}) \|_{L^2(E) \to L^2(F)} \lesssim \max \left\{1-\frac{d(E,F)}{r},0\right\},
\end{equation}
and
\begin{equation} \label{eq:L2OD-fs2}
 \|  r\Gamma \varphi(r\sqrt{L}) \|_{L^2(E) \to L^2(F)} \lesssim \max \left\{1-\frac{d(E,F)}{r},0\right\}.
\end{equation}
Assume in addition  $(G_{p_0})$ for some $p_0\in(2,\infty]$. Then for every $p\in(2,p_0)$, we have $L^2$-$L^p$ off-diagonal estimates: for every pair of balls $B_1,B_2$ of radius $r>0$, we have
$$ \|  r\Gamma \varphi(r\sqrt{L}) \|_{L^2(B_1) \to L^p(B_2)} \lesssim |B_1|^{\frac{1}{p}-\frac{1}{2}}  \max \left\{1-\frac{d(B_1,B_2)}{r},0\right\}.$$
\end{lemma}

\begin{proof} 
The first statement on $\varphi(r\sqrt{L})$ and the estimate \eqref{eq:L2OD-fs} follow from the Fourier inversion formula and the bounded Borel functional calculus of $\sqrt{L}$, see \cite[Lemma 4.4]{AMcM}. The estimate \eqref{eq:L2OD-fs2} then follows from \eqref{eq:L2OD-fs}, the assumption $(G_2)$ and the fact that $\Gamma$ is supposed to be local. 
For the proof of $L^2$-$L^p$ off-diagonal estimates for $\varphi(r\sqrt{L})$, we use that thanks to real time off-diagonal estimates and the analyticity of $(e^{-tL})_{t>0}$ in $L^q(X,\mu)$, $q \in (p_L,p^L)$, we have complex time $L^2$-$L^p$ off-diagonal estimates for the semigroup. The $L^2$-$L^p$ estimates on $\varphi(r\sqrt{L})$ then again follow from the Fourier inversion formula. 
Now assume $(G_{p_0})$ for some $p_0\in(2,\infty]$. Under this assumption, one has complex time $L^2$-$L^p$ estimates for $(\sqrt{t}\Gamma e^{-tL})_{t>0}$, see e.g. \cite[Proposition 3.16]{memoirs} for a proof. One can once more apply the Fourier inversion formula, and obtains the last estimate of the lemma. 
\end{proof}

We may now prove the following version of Lemma \ref{lemma}.

\begin{lemma} \label{lemma-bis}
Suppose  $p^L>\nu$, and let $p \in [2,p^L)$. Let $\varphi \in \mathcal{S}(\R)$ be an even function with $\supp \hat{\varphi} \subseteq [-1,1]$.
Then for every ball $B$ of radius $r>0$ and every function $f\in L^2(X,\mu)$, we have
\begin{align*}
 	\left(\aver{B} |r\sqrt{L}\varphi(r\sqrt{L}) f|^{2} \, d\mu \right)^{1/2}  \lesssim \sup_{Q\supset B} \left(\aver{Q} |f|^2 \, d\mu \right)^{1/2},
\end{align*}
where the supremum is taken over all balls $Q$ containing $B$.
\end{lemma}

\begin{proof}
Let $p\in(\nu,p^L)$. Write for $M \in \N$ with $M>\nu$
\begin{align*}
	z\varphi(z) = z(1+z^2)^{-M}(1+z^2)^M \varphi(z).
\end{align*}
Note that under our assumptions on $\varphi$, for every $K \in \N$ the function $z \mapsto z^{2K} \varphi(z) \in \mathcal{S}(\R)$ is  even with compact Fourier support. We can therefore apply Lemma \ref{lemma0} to $(1+z^2)^M\varphi(z)$ and get $L^2$-$L^2$ off-diagonal estimates for 
$(I+r^2L)^M\varphi(r\sqrt{L})$ of the form \eqref{eq:L2OD-fs}. On the other hand, we know that with $(I+r^2L)^{-1}$ also $(r^2L)^{1/2}(I+r^2L)^{-M}$ satisfies $L^2$-$L^p$ off-diagonal estimates of order $N:=\frac{1}{2}+\frac{\nu}{2}(\frac{1}{2}-\frac{1}{p})$, see e.g. \cite[Proposition 5.3]{AKMP}. Combining the two estimates gives $L^2$-$L^p$ off-diagonal estimates of order $N$ for $r\sqrt{L}\varphi(r\sqrt{L})$. We therefore get
\begin{align*}
	\left(\aver{B}|r\sqrt{L}\varphi(r\sqrt{L})f|^2\,d\mu\right)^{1/2}
	&\leq 
	\left(\aver{B}|r\sqrt{L}\varphi(r\sqrt{L})f|^p\,d\mu\right)^{1/p}\\
	&\lesssim \sum_{j \geq 0} 2^{-2jN} |B|^{-1/2} \left(\int_B|f|^2\,d\mu\right)^{1/2}\\
	& \lesssim \sum_{j \geq 0} 2^{-j(2N-\frac{\nu}{2})}  \left(\aver{2^jB}|f|^2\,d\mu\right)^{1/2}\\
	& \lesssim  \sup_{Q\supset B} \left(\aver{Q} |f|^2 \, d\mu \right)^{1/2},
\end{align*}
where we used $N>\frac{\nu}{4}$ (since $p>\nu$) in the last estimate. 
\end{proof}

\subsection{Poincar\'e inequalities and gradient estimates} \label{subsec:ppgp}

\begin{proposition}[$L^p$ Caccioppoli inequality] \label{prop:Caccioppoli} Let $(X,\mu,\Gamma,L)$ as in Section \ref{sec11} with $L$ self-adjoint. Assume $(G_{p})$ for some $p\in[2,+\infty]$. Then for every $q\in(p_L,p]$ and every integer $N$,
 \begin{equation}\label{cacp}
 r\left(\aver{B_r} | \Gamma f|^q d\mu \right)^{1/q} \lesssim \left(\aver{2B_r} |f|^q\, d\mu \right)^{1/q}+ \left(\aver{2B_r} | (r^2 L)^N f|^q \, d\mu \right)^{1/q}\end{equation}
for all $f\in {\mathcal D}(L^N)$ and all balls  $B_r$ of radius $r$.
\end{proposition}

The result was shown in \cite[Proposition 5.4]{BCF1} in the case $N=1$. The proof for general $N \in \N$ is similar, we give a proof here for the sake of completeness.

\begin{proof}
Consider an even function $\varphi \in \mathcal{S}(\R)$ with $\supp \hat{\varphi} \subseteq [-1,1]$ and $\varphi(0)=1$. Consequently, $\varphi'(0)=0$, and $z\mapsto z^{-1}\varphi'(z) \in \mathcal{S}(\R)$ is even with Fourier support in $ [-1,1]$, cf. \cite[Lemma 6.1]{AMcM}.  Fix a ball $B$ of radius $r>0$, an exponent $q\in(1,p]$ and split
$$ f = \varphi(r\sqrt{L}) f + (I- \varphi(r\sqrt{L})) f.$$
Since $\varphi(0)=1$, one has
\begin{align*}
 (I- \varphi(r\sqrt{L})) = \int_0^r \sqrt{L}\varphi'(s\sqrt{L})  \, ds.
\end{align*}
Using the finite propagation speed  property applied to the functions $\varphi$ and $z\mapsto z^{-1+2N}\varphi'(z)$, we have that both $\varphi(r\sqrt{L})$ and $(r^2 L)^{-N}(1- \varphi(r\sqrt{L}))$ satisfy the  propagation property at a speed $1$ and so propagate at a distance at most $r$. The same stills holds by composing with the gradient. Hence, 
\begin{align}
	\|\Gamma f \|_{L^q(B)}
		& \lesssim \|\Gamma \varphi(r \sqrt{L})\|_{q\to q} \|f\|_{L^q(2B)} \nonumber  \\
		& +  \|\Gamma (1- \varphi(r\sqrt{L}))(r^2 L)^{-N} \|_{q\to q} \| (r^2L)^N f \|_{L^q(2B)}. \label{eq:cc}
\end{align}

For $q\in(2,p]$, one obtains $(G_q)$ via interpolation between $(G_2)$ and $(G_p)$. For $q\in(p_L,2)$, $(G_q)$ is a consequence of $(G_2)$. By writing the resolvent via the Laplace transform as $(1+r^2L)^{-1} = \int_0^{+\infty}  e^{-t(1+r^2L)} dt$, we deduce gradient bounds for the resolvent in $L^q$, that is
\begin{align*}
\| \Gamma (1+r^2L)^{-1} \|_{q \to q} & \lesssim \int_0^{+\infty }e^{-t} \|\Gamma e^{-tr^2L}\|_{q \to q} \,dt 
  \lesssim \int_0^{+\infty} \frac{e^{-t}}{r\sqrt{t}} \,dt  \lesssim r^{-1}. 
\end{align*}
Denote $\psi:=\varphi$ or $\psi:=x\mapsto (1-\varphi(x))/x^{2N}$, and consider $\lambda(x)=\psi(x) (1+x^2)$. Hence
$$ \psi(r \sqrt{L}) = (1+r^2L)^{-1} \lambda(r\sqrt{L}) $$
and therefore
\begin{align*} 
\| \Gamma \psi(r \sqrt{L})\|_{q\to q} & \leq \| \Gamma (1+r^2L)^{-1} \|_{q \to q} \| \lambda(r\sqrt{L}) \|_{q \to q} 
 \lesssim r^{-1} \| \lambda(r\sqrt{L}) \|_{q \to q}.
\end{align*}
Observe that $\lambda \in \mathcal{S}(\R)$ since $\varphi(0)=1$ and $\varphi'(0)=0$. 
By a functional calculus result (see e.g. \cite{Blunck}), we then have 
$$ \sup_{r>0} \ \| \lambda(r\sqrt{L}) \|_{q \to q} \lesssim 1,$$
and consequently,
$$ \| \Gamma \psi(r \sqrt{L}) \|_{q\to q} \lesssim r^{-1}.$$
Coming back to \eqref{eq:cc}, we obtain
\begin{align*}
	\| \Gamma f  \|_{L^q(B)}	& \lesssim r^{-1} \| f\|_{L^q(2B)} +  r \| L f \|_{L^q(2B)}. 
	\end{align*}
\end{proof}

We are now going to give  a more general version of the main result of \cite{AC}.
More precisely,  \cite[Theorem 0.4]{AC} states that if $X$ is a complete non-compact Riemannian manifold satisfying the doubling condition \eqref{d} and $L$ is its nonnegative Laplace-Beltrami operator,  then under the Poincar\'e inequality $(P_2)$ there exists $\eps>0$ such that $(R_p)$ holds for $p\in[2,2+\eps)$. This result relies on the self-improvement of Poincar\'e inequalities from \cite{KZ}, but also on considerations on the Hodge projector that are specific to the Riemannian setting. We give here a proof that is valid in our more general setting and also gives a $L^{p_0}$-version.

\begin{theorem}\label{ac}
 Let $(X,\mu,\Gamma,L)$ be as in Section \ref{sec11} with $L$ self-adjoint and $L$ and $\Gamma$ satisfying the conservation property. If for some $p_0\in[2,\infty)$, the combination $(P_{p_0})$ and $(G_{p_0})$ holds, then there exists $\eps>0$ such that $(G_p)$ holds for $p\in[p_0,p_0+\eps)$.
\end{theorem}

\begin{proof}  
Let $N \in \N$ to be chosen later. Consider $f\in L^{p_0}(X,\mu)$, $t>0$ and set $g:=e^{-tL}f$.
For every ball $B_r$ with radius $r>0$,  the Caccioppoli inequality \eqref{cacp} together with \eqref{d} yields
\begin{equation*}
 \left(\aver{B_r} |\Gamma g|^{p_0} d\mu \right)^{1/p_0} \lesssim \frac{1}{r} \left(\aver{2B_r} \left|g- \aver{2B_r} g \, d\mu \right|^{p_0}\, d\mu \right)^{1/p_0} + 
  r^{-1} \left( \aver{2B_r} |(r^2L)^N g|^{p_0} d\mu \right)^{1/p_0}.\end{equation*}
We know by \cite{KZ} that   $(P_{p_0})$ implies $(P_{p_0-\kappa})$ for  $\kappa>0$ small enough which self-improves into a Sobolev-Poincar\'e inequality $(P_{p_0,p_0-\kappa})$ (see Lemma \ref{met}). Hence, for $\kappa$ small enough
$$  \left(\aver{2B_r} \left|g- \aver{2B_r} g \, d\mu \right|^{p_0}\, d\mu \right)^{1/p_0} \lesssim r \left(\aver{2B_r} |\Gamma g|^{p_0-\kappa} d\mu \right)^{1/(p_0-\kappa)} .$$
We obtain in this way
\begin{equation}
\label{gg-2}\left(\aver{B_r} |\Gamma g|^{p_0} \, d\mu \right)^{1/p_0} \lesssim \left(\aver{2B_r} |\Gamma g|^{p_0-\kappa} \, d\mu \right)^{1/(p_0-\kappa)} + r^{-1}  \left( \aver{2B_r} |(r^2 L)^N g|^{p_0} \, d\mu \right)^{1/p_0}.
\end{equation}
 Now, for a ball $B_{\sqrt{t}}$  of radius $\sqrt{t}$  and a parameter $K>\nu$  to be chosen later, consider the function
$$ h := |\Gamma g| + c(f,B_{\sqrt{t}}),$$
where 
$$ c(f,B_{\sqrt{t}}):=t^{-1/2} \inf_{x\in B_{\sqrt{t}}}  {\mathcal M}_2(f)(x)$$
is a constant. Since $g=e^{-tL}f$, for $N$ sufficiently large Lemma \ref{lemma} implies that for every ball $B_r\subset B_{\sqrt{t}}$ with radius $r$,  
$$r^{-1}  \left( \aver{2B_r} |(r^2L)^N e^{-tL} f|^{p_0} \, d\mu \right)^{1/p_0}  \lesssim  c(f,B_{\sqrt{t}}).$$
Therefore  \eqref{gg-2} yields
\begin{eqnarray}\label{dalmon-2}
 \left(\aver{B_r} |\Gamma g|^{p_0} d\mu \right)^{1/p_0}  &\lesssim& \left(\aver{2B_r} |\Gamma g|^{p_0-\kappa} d\mu \right)^{1/(p_0-\kappa)}   +  c(f,B_{\sqrt{t}}).
\end{eqnarray}
It follows that
\begin{equation} \label{eq:ouf-2}
\left(\aver{B_r} h^{p_0} \,d\mu \right)^{1/p_0} \lesssim \left(\aver{2B_r} h^{p_0-\kappa} \,d\mu \right)^{1/(p_0-\kappa)}.  \end{equation}

Now for $x\in B_{\sqrt{t}}$  consider the quantities
$$ F_{p_0}(x):=\sup_{x\in B\subset 2 B_{\sqrt{t}}} \left(\aver{B} h^{p_0} \, d\mu \right)^{1/p_0}$$
and
$$ F_{p_0-\kappa}(x):=\sup_{x\in B\subset 2B_{\sqrt{t}}} \left(\aver{B} h^{p_0-\kappa} d\mu \right)^{1/(p_0-\kappa)},$$
where the supremum is taken over all the balls $B$ containing $x$ and included into $B_{\sqrt{t}}$.
Then, we have
$$ F_{p_0}(x) \lesssim F_{p_0-\kappa}(x).$$
Indeed, for $B\subset 2B_{\sqrt{t}}$ a ball containing $x$, if $2B \subset 2B_{\sqrt{t}}$ then we may apply \eqref{eq:ouf-2}, and if $2B$ is not included in $2B_{\sqrt{t}}$ (since $x\in B \cap B_{\sqrt{t}}$) then $r(B)\simeq \sqrt{t}$ and so we directly apply off-diagonal estimates to have
$$ \left(\aver{B_r} h^{p_0} \,d\mu \right)^{1/p_0} \lesssim c(f,B).$$

We may apply Gehring's Lemma (Theorem \ref{thm:gehring}), and we have for some $\eps>0$  that
$$ \left(\aver{B_{\sqrt{t}}} |h|^p d\mu \right)^{1/p} \lesssim \left(\aver{2B_{\sqrt{t}}} |h|^{p_0} d\mu \right)^{1/p_0}$$
for every $p\in(p_0,p_0+\eps)$.
Hence, with \eqref{eq:ouf-2} applied for $r=\sqrt{t}$, it follows that
\begin{equation} \left(\aver{B_{\sqrt{t}}} |h|^p \, d\mu \right)^{1/p} \lesssim \left(\aver{4B_{\sqrt{t}}} |h|^{p_0-\kappa} \, d\mu \right)^{1/(p_0-\kappa)}. \label{eq:fin} \end{equation}
Hence, we deduce that 
$$ \left(\aver{B_{\sqrt{t}}} |\Gamma e^{-tL}f|^p \, d\mu \right)^{1/p} \lesssim \left(\aver{4 B_{\sqrt{t}}} |\Gamma e^{-tL}f|^{p_0-\kappa} \, d\mu \right)^{1/(p_0-\kappa)}+c(f,B_{\sqrt{t}}).$$
Interpolating the Davies-Gaffney estimates  \eqref{eq:DG} with $(G_{p_0})$ yields $L^{2}-L^{p_0-\kappa}$ off-diagonal estimates for $\sqrt{t} \Gamma e^{-tL}$ and so for a large enough integer $d\geq 1$
$$ \left(\aver{B_{\sqrt{t}}} |\sqrt{t} \Gamma e^{-tL}f|^p \, d\mu \right)^{1/p} \lesssim \inf_{x\in B_{\sqrt{t}}}  {\mathcal M}_2(f)(x) \lesssim \left(\aver{B_{\sqrt{t}}} {\mathcal M}_2(f)^p \, d\mu \right)^{1/p}.$$
By summing over a covering of $X$ by balls with radius $\sqrt{t}$ with the $L^p$-boundedness of the maximal function, we then deduce that $\sqrt{t} \Gamma e^{-tL}$ is bounded on $L^p$, uniformly with respect to $t>0$, which is $(G_p)$ as desired. 
\end{proof}

\begin{rem} Using the self-improvement of reverse H\"older inequality \cite[Appendix B]{BCF1}, \eqref{eq:ouf-2} yields for any $\rho\in(1,2)$
$$\left(\aver{B_r} h^{p_0} \,d\mu \right)^{1/p_0} \lesssim \left(\aver{2B_r} h^{\rho} \,d\mu \right)^{1/\rho}.$$
So a careful examination of the previous proof shows that we have in fact the following inequality: for every ball $B_r$ of radius $r>0$ and every $g\in{\mathcal D}(L^N)$,
$$ \left(\aver{B_r} |\Gamma g|^{p_0} \, d\mu \right)^{1/p_0} \lesssim \left(\aver{2B_r} |\Gamma g|^{\rho} \, d\mu \right)^{1/\rho} + r^{-1} \| (r^2 L)^N g\|_{L^\infty(4B_r)}.$$
\end{rem}

By combining Theorem \ref{thmbis} (which will be proved later in a more general setting with a sectorial operator $L$) and Theorem \ref{ac} with the fact that $(R_p)$ always implies $(G_p)$ by $L^p$ analyticity of the semigroup and the fact that the Poincar\'e inequality is weaker and weaker as $p$ increases, we deduce the following statement, which encompasses Theorems \ref{thmbis} and \ref{ac} and extends both \cite[Theorem 1.3]{ACDH} and \cite[Theorem 0.4]{AC}.

\begin{theorem} \label{thm:main2} Let $(X,\mu,\Gamma,L)$ be as in Section \ref{sec11}  with $L$ self-adjoint and $\Gamma$ and $L$ satisfying the conservation property.
If for some $p_0\in[2,\infty]$, the combination $(P_{p_0})$ with $(G_{p_0})$ holds, then there exists $p_1\in(p_0,\infty]$ such that
$$ (p_L,p_1) = \{p\in(p_L,\infty), \ (G_p) \textrm{ holds}\} = \{p\in(p_L,\infty), \ (R_p) \textrm{ holds}\}.$$
\end{theorem}

\subsection{Reverse H\"older property and Poincar\'e inequalities}

First let us recall the following result about $L^2$ Poincar\'e inequalities, proved in \cite[Theorem 6.1]{BCF1} combined with Theorem \ref{thm:wpoincare}.

\begin{theorem} \label{thm:poincare} Let $(X,\mu,\Gamma,L)$  be  as in Section \ref{sec11} with $L$ self-adjoint. Suppose that for every ball $B$ of radius $r>0$ and every function $u\in {\mathcal F}$ harmonic on $2B$, we have
$$ \left( \aver{B} | u- \aver{B} u \, d\mu |^2 \, d\mu \right)^{1/2} \lesssim r \left( \aver{B} |\Gamma u|^2\, d\mu \right)^{1/2}.$$
Then the Poincar\'e inequality $(P_2)$ holds.
\end{theorem}

From this result we obtain

\begin{coro} \label{coro:rhp-poinc}
Let $(X,\mu,\Gamma,L)$  be as in Section \ref{sec11} with $L$ self-adjoint.  Assume \eqref{RHp} for some $p_0\in(2,\infty)$. Then  $(P_p)$ for some $p \in (2,p_0]$ implies $(P_2)$.
\end{coro}

\begin{proof} By $(P_p)$ which implies $(P_{p_0})$, we have for every ball $B$ of radius $r>0$ and every function $u \in{\mathcal F}$ which is harmonic on $2B$ 
$$ \left(\aver{B} \left| u -\aver{B} u \,d\mu \right|^{p_0} \, d\mu \right)^{1/p_0} \lesssim  r \left( \aver{B} |\Gamma u|^{p_0} \, d\mu \right)^{1/p_0}.$$
So we deduce with \eqref{RHp} that
\begin{align*} 
\left(\aver{B} \left| u -\aver{B} u \,d\mu \right|^{2} \, d\mu \right)^{1/2} & \leq \left(\aver{B} \left| u -\aver{B} u \,d\mu \right|^{p_0} \, d\mu \right)^{1/p_0} \lesssim  r \left( \aver{B} |\Gamma u|^{p_0} \, d\mu \right)^{1/p_0} \\
& \lesssim r \left( \aver{2B} |\Gamma u|^{2} \, d\mu \right)^{1/2}.
\end{align*}
By Theorem \ref{thm:wpoincare}, we obtain $(P_2)$  for every harmonic function. The full Poincar\'e inequality $(P_2)$ then follows by Theorem \ref{thm:poincare}.
\end{proof}

\begin{proposition} \label{prop:pprp} Let $(X,\mu,\Gamma, L)$  be as in Section \ref{sec11} with $L$ self-adjoint and $\Gamma$ and $L$ satisfying the conservation property.  Assume $(G_{p_0})$ for some $p_0\in(2,\infty)$. Then $(P_{p_0})$ yields \eqref{RHp}.
\end{proposition}

\begin{proof}
Let $B$ be a ball of radius $r>0$,  and $u\in {\mathcal D}$ harmonic on $2B$. Then Proposition \ref{prop:Caccioppoli} yields that $(G_{p_0})$ implies a $L^{p_0}$-Caccioppoli inequality, so using the conservation, it follows that
$$ \left( \aver{B} |\Gamma u|^{p_0} \, d\mu \right)^{1/p_0} \lesssim \frac{1}{r} \left( \aver{2B} \left| u - \left(\aver{2B} u \, d\mu \right)\right|^{p_0} \, d\mu \right)^{1/p_0}.$$
By the self-improvement of Poincar\'e inequalities (see \cite{KZ}), there exists $\epsilon>0$ such that $(P_{p_0-\epsilon})$ holds and so for a small enough $\epsilon>0$ the Sobolev-Poincar\'e inequality (Lemma \ref{met}) yields
$$ \left( \aver{B} |\Gamma u|^{p_0} \, d\mu \right)^{1/p_0} \lesssim \left( \aver{2B} |\Gamma u|^{p_0-\epsilon} \, d\mu \right)^{1/(p_0-\epsilon)}.$$
Then the exponent of the RHS of such a reverse H\"older inequality can always be improved (see \cite[Theorem B.1]{BCF1}), so that we deduce 
$$ \left( \aver{B} |\Gamma u|^{p_0} \, d\mu \right)^{1/p_0} \lesssim \left( \aver{2B} |\Gamma u|^{2} \, d\mu \right)^{1/2},$$
which concludes the proof of \eqref{RHp}.
\end{proof}

In the Riemannian setting with $L$ denoting the Laplace-Beltrami operator, it is already known that the Poincar\'e inequality $(P_2)$ implies a reverse H\"older property $(RH_{p})$ for exponents $p>2$ sufficiently close to $2$, see \cite[Proposition 2.2]{AC}. We extend  this result here (with a slightly different proof) for the sake of completeness. We are considering the setting of Dirichlet forms with a carr\'e du champ $\Gamma$ (see the first example of Subsection \ref{subsec:ex}), see also \cite{GSC} and \cite{BCF1} for details.

\begin{proposition} \label{prop:P2RH} Let $(X,\mu,\Gamma,L)$ be as in Section \ref{sec11}, and assume that $L$ generates a strongly local and regular Dirichlet form such that $(X,d,\mu,{\mathcal E})$ is a doubling metric  measure   Dirichlet space with a carr\'e du champ $\Gamma$ and that \eqref{FKR} and $(P_2)$ are satisfied. Then there exists $\epsilon>0$ such that $(RH_p)$ holds for every $p\in(2,2+\epsilon)$.
\end{proposition}

\begin{rem} Note that the proof only requires the conservation property for $\Gamma$ and not for the operator $L$.
\end{rem}

\begin{proof}
Let $B_0=B(x_0,r_0)$ be a ball and $f\in{\mathcal D}$ be a function, harmonic on $2B_0$. Then for every sub-ball $B\subset B_0$ of radius $r\leq r_0$, consider a Lipschitz function $\chi_B$, supported on $2B$, equal to $1$ on $B$ with $\|\Gamma \chi_B\|_\infty \lesssim r^{-1}$.
We have (by using the Dirichlet form structure and Leibniz inequality for carr\'e du champ $\Gamma$), since $\chi_B$ is supported on $2B_0$ where $f$ is harmonic
$$ \int \chi_B^2 |\Gamma f|^2 \, d\mu \leq \int \left|f-\aver{B} f d\mu\right| \chi_B \left|\Gamma \chi_B\right| |\Gamma f| \, d\mu.$$
Consider $q>2$ such that $\max(0, \frac{1}{2}-\frac{1}{\nu}) < \frac{1}{q} < \frac{1}{2}$. Then by Lemma \ref{met}, we know that $(P_2)$ implies a $L^2$-$L^q$ Sobolev-Poincar\'e inequality. H\"older's inequality (with $q'$ the conjugate exponent of $q$) yields
$$ \int \chi_B^2 |\Gamma f|^2 \, d\mu \lesssim \left(\int \chi_B^{q'} |\Gamma f|^{q'} \, d\mu\right)^{1/q'} r^{-1} \left( \int_{2B} \left|f-\aver{B} f d\mu \right|^q \, d\mu\right)^{1/q}.$$
Using the $L^2$-$L^q$ Sobolev-Poincar\'e inequality, this implies
\begin{equation} \int \chi_B^2 |\Gamma f|^2 \, d\mu \lesssim \left(\int \chi_B^{q'} |\Gamma f|^{q'} \, d\mu\right)^{1/q'} \left(\int_{2B} |\Gamma f|^2 \, d\mu\right)^{1/2}. \label{eq:boule} \end{equation}
Then let us fix a ball $B_1 \subset B_0$ and consider $(Q^j)_j$ a Whitney covering of $B_1$, which is a collection of balls satisfying in particular
 $$ \textrm{$(Q^j)_j$ is a covering of $B_1$ and $(2Q^j)_j$ is also a bounded covering of $B_1$.} $$
So \eqref{eq:boule} applied to each $Q^j$ implies 
$$ \int_{Q^j} |\Gamma f|^2 \, d\mu \lesssim \left(\int_{2Q^j} |\Gamma f|^{q'} \, d\mu\right)^{1/q'} \left(\int_{2Q^j} |\Gamma f|^2 \, d\mu\right)^{1/2}.$$
By summing over $j$, Cauchy-Schwarz's inequality gives
$$ \int_{B_1} |\Gamma f|^2 \, d\mu \lesssim\left( \sum_{j} \left(\int_{2Q^j} |\Gamma f|^{q'} \, d\mu\right)^{2/q'}\right)^{1/2} \left(\int_{B_1} |\Gamma f|^2 \, d\mu\right)^{1/2}.$$
Consequently, since $q'<2$ (so $\ell^{q'} \subset \ell^2$), we deduce that
$$ \left(\int_{B_1} |\Gamma f|^2 \, d\mu\right)^{1/2} \lesssim \left(\int_{B_1} |\Gamma f|^{q'} \, d\mu\right)^{1/q'}.$$
This property holds for every sub-ball $B_1 \subset B_0$. We may then apply Gehring's argument: there exists $\epsilon>0$ such that
$$ \left(\int_{B_0} |\Gamma f|^p \, d\mu\right)^{1/p} \lesssim \left(\int_{2B_0} |\Gamma f|^{2} \, d\mu\right)^{1/2},$$
uniformly with respect to the ball $B_1$ for every $p\in(2,2+\epsilon)$.
\end{proof}

\subsection{Boundedness of Riesz transforms in a self-adjoint setting} \label{subsec:r}

Let us now focus on the situation where $L$ does not satisfy the conservation property. In this case, the use of Poincar\'e inequalities does not seem to help and we aim to explain how the reverse H\"older inequalities $(RH_p)$ can be used to get around this difficulty.

\begin{theorem}\label{thm:rp} Let $(X,\mu,\Gamma, L)$ be as in Section \ref{sec11} with $L$ self-adjoint. If $p^L\leq \nu$, assume in addition  \eqref{eq:RR2}. If $(G_{p_0})$ and \eqref{RHp} hold for some $p_0\in(2,p^L)$, then $(R_p)$ is satisfied for every $p\in (2,p_0)$.
\end{theorem}

\begin{proof}
We first assume $p^L>\nu$. 
We will apply Proposition \ref{prop:extrapolation} to the Riesz transform.
Consider an even function $\varphi \in \mathcal{S}(\R)$ with $\supp \hat{\varphi} \subseteq [-1,1]$ and $\varphi=1$ on $[-1/2,1/2]$. 
Let $q\in(2,p_0)$. Let $f \in L^2(X,\mu)$, and let $B$ be a ball of radius $r>0$.

We first check \eqref{eq:ass1} with $A_{r}:=\varphi(r\sqrt{L})$ and $T={\mathcal R}$. We follow the arguments of \cite[Lemma 3.1]{ACDH}, adapted to our approximation operators $A_{r}$. 
We use the integral representation
$$ I-\varphi(r\sqrt{L}) = \int_0^r \sqrt{L} \varphi'(s\sqrt{L}) \, ds$$ to write 
$$
 \left|T(I-\varphi(r\sqrt{L}))f \right| \leq \int_0^r \left|\Gamma \varphi'(s\sqrt{L})f \right| \, ds.
$$
Since $\varphi'$ is supported in $[-1,-1/2] \cup [1/2,1]$, we may repeat similar arguments as in \cite{ACDH} to have $L^2$-$L^2$ off-diagonal estimates for $\Gamma \varphi'(s\sqrt{L})$ and then obtain with the $L^2$-boundedness of $T$ and \eqref{eq:L2OD-fs} that
\begin{align} \label{eq:eq22}
 \left(\aver{B} | T(I-\varphi(r\sqrt{L})) f|^2\, d\mu \right)^{1/2} & \lesssim  \inf_{x\in B} \calM_2(f)(x),
\end{align}
which concludes the proof of \eqref{eq:ass1}.

Let us now check \eqref{eq:ass2}, again with $T={\mathcal R}$ and $A_{r}=\varphi^2(r\sqrt{L})$.
Let $h:=\varphi(r\sqrt{L})L^{-1/2}f$. By Lemma \ref{lem:LM-ter}, there exists $u\in {\mathcal F}$ such that $h-u\in {\mathcal F}$ is supported in the ball $2B$ and $u$ is harmonic in $2B$ with
\begin{equation} \left(\aver{2B} |\Gamma(h-u)|^2 \, d\mu \right)^{1/2} + \left(\aver{2B} |\Gamma u|^2 \, d\mu \right)^{1/2} \lesssim \left(\aver{2B} |\Gamma h|^2 \, d\mu \right)^{1/2}
+ r\left( \aver{2B} |Lh|^2 \, d\mu \right)^{1/2}. 
\label{eq:grad-bis-2} \end{equation}
Then we split, with $h=\varphi(r\sqrt{L})L^{-1/2}f$,
\begin{align} |\Gamma L^{-1/2} \varphi^2(r\sqrt{L}) f| & = |\Gamma \varphi(r\sqrt{L}) h| \nonumber \\
 & \leq |\Gamma \varphi(r\sqrt{L}) (h-u)| + |\Gamma (I-\varphi(r\sqrt{L}))u| + |\Gamma u|.  \label{split}
\end{align}
For the first quantity, we use $L^2$-$L^p$ off diagonal estimates of $\Gamma \varphi(r\sqrt{L})$ (see Lemma \ref{lemma0}) with the support condition of $h-u$ and the Faber-Krahn inequality \eqref{FKR} to obtain
\begin{align*}
\left( \aver{B} \left|\Gamma \varphi(r\sqrt{L}) (h-u) \right|^p \, d\mu \right)^{1/p} & \lesssim r^{-1} \left( \aver{2B} \left| h-u\right|^2 \, d\mu \right)^{1/2} \\
 & \lesssim  \left( \aver{2B} \left| \Gamma (h-u)\right|^2 \, d\mu \right)^{1/2} \\
   &\lesssim \left( \aver{2B} \left| \Gamma h\right|^2 \, d\mu \right)^{1/2}
  + r\left( \aver{2B} |Lh|^2 \, d\mu \right)^{1/2},
 \end{align*}
where we have used \eqref{eq:grad-bis-2} in the last step.

The second quantity in \eqref{split} is equal to $0$. To see this, write 
$$ (I-\varphi(r\sqrt{L}))u = \int_0^r \sqrt{L} \varphi'(s\sqrt{L}) u \, ds =  \int_0^r L^{-1/2} \varphi'(s\sqrt{L}) L u \, ds.$$
Since $\varphi$ is an even function with its spectrum included in $[-1,1]$, also $x\mapsto x^{-1} \varphi'(x)$ is even with its spectrum still included in $[-1,1]$. By the finite propagation speed  property, we deduce that $(sL)^{-1/2} \varphi'(s\sqrt{L})$ propagates at a distance at most $s \leq r$. Since $u$ is harmonic on $2B$ then we deduce that 
$$ (sL)^{-1/2} \varphi'(s\sqrt{L}) Lu =0 \qquad \textrm{on $B$.}$$
Therefore $(I-\varphi(r\sqrt{L}))u=0$ on $B$, and by the locality property of the gradient, 
$$ |\Gamma (I-\varphi(r\sqrt{L}))u|=0 \qquad \textrm{on $B$.}$$

It remains to study the last term in \eqref{split}. Using the reverse H\"older property \eqref{RHp} and that $u$ is harmonic on $2B$ gives
$$ \left(  \aver{B} |\Gamma u|^p \, d\mu \right)^{1/p} \lesssim  \left(\aver{2B} |\Gamma u|^2 \, d\mu \right)^{1/2} \lesssim \left(\aver{2B} |\Gamma h|^2 \, d\mu \right)^{1/2}
+ r\left( \aver{2B} |Lh|^2 \, d\mu \right)^{1/2},$$
where we have used again \eqref{eq:grad-bis-2}.

Consequently, we have obtained
\begin{align} 
 & \left(  \aver{B} |\Gamma L^{-1/2} \varphi(r\sqrt{L})f |^p \, d\mu \right)^{1/p}  \label{eq:sum-riesz} \\ \nonumber
 &\qquad  \lesssim \left(  \aver{2B} |\Gamma L^{-1/2} \varphi(r\sqrt{L}) f|^2 \, d\mu \right)^{1/2} + r\left( \aver{2B} |\sqrt{L} \varphi(r\sqrt{L}) f|^2 \, d\mu \right)^{1/2} \\ \nonumber
  & \qquad \lesssim \left(  \aver{2B} |\Gamma L^{-1/2} (\varphi(r\sqrt{L})-I) f|^2 \, d\mu \right)^{1/2} +  \left(  \aver{2B} |\Gamma L^{-1/2} f|^2 \, d\mu \right)^{1/2} \nonumber \\
  & \qquad  \qquad  + \left( \aver{2B} |r\sqrt{L} \varphi(r\sqrt{L}) f|^2 \, d\mu \right)^{1/2}. \nonumber
\end{align}
The first term is bounded by $\inf_{x\in B} {\mathcal M}_2(f)(x)$ by \eqref{eq:eq22}.
So by Lemma \ref{lemma-bis} with the assumption $p^L>\nu$, we then conclude
$$ \left(  \aver{B} |\Gamma L^{-1/2} \varphi(r\sqrt{L})f |^p \, d\mu \right)^{1/p} \lesssim  \left(  \aver{2B} |\Gamma L^{-1/2} f|^2 \, d\mu \right)^{1/2} + \inf_{x\in B} {\mathcal M}_2(f)(x),$$
which implies \eqref{eq:ass2}.

We may then apply Proposition \ref{prop:extrapolation} to the Riesz transform and conclude the proof of $(R_p)$ for every $p\in[2,p_0)$.

If otherwise $p^L\leq \nu$, the assumption \eqref{eq:RR2} implies that \eqref{eq:grad-bis-2} can be improved to 
\begin{equation*} \left(\aver{2B} |\Gamma(h-u)|^2 \, d\mu \right)^{1/2} + \left(\aver{2B} |\Gamma u|^2 \, d\mu \right)^{1/2} \lesssim \left(\aver{2B} |\Gamma h|^2 \, d\mu \right)^{1/2}.
\end{equation*}
This in turn improves \eqref{eq:sum-riesz} to
\begin{align*} 
 & \left(  \aver{B} |\Gamma L^{-1/2} \varphi(r\sqrt{L})f |^p \, d\mu \right)^{1/p} \\
 & \lesssim \left(  \aver{B} |\Gamma L^{-1/2} \varphi(r\sqrt{L}) f|^2 \, d\mu \right)^{1/2} \\
  & \lesssim \left(  \aver{B} |\Gamma L^{-1/2} (\varphi(r\sqrt{L})-I) f|^2 \, d\mu \right)^{1/2} +  \left(  \aver{2B} |\Gamma L^{-1/2} f|^2 \, d\mu \right)^{1/2}\\
& \lesssim  \left(  \aver{2B} |\Gamma L^{-1/2} f|^2 \, d\mu \right)^{1/2} + \left(  \aver{4B} |f|^2 \, d\mu \right)^{1/2},
 \label{eq:sum-riesz}
  \end{align*}
where in the last step, we use the global $L^2$-boundedness of $\Gamma L^{-1/2} (\varphi(r\sqrt{L})-I)$ and the local property of $\Gamma$ together with the finite  propagation speed property.
\end{proof}

In the same spirit, we also improve one of the main result of \cite{AC}.

\begin{proposition} \label{prop:RHG} Let $(X,\mu,\Gamma,L)$ be as in Section \ref{sec11} with $L$ self-adjoint and $p^L>\nu$. Assume that for some $q>2$ a reverse H\"older property $(RH_{q})$ is satisfied. Then there exists $\epsilon>0$ such that $(G_{p})$ holds for $p\in(2,2+\epsilon)$.
\end{proposition}

We do not detail the proof of this proposition here, since in Section \ref{sec:appli} we will prove  a stronger result (involving also an elliptic perturbation of the considered operator), see Proposition \ref{prop:RHGbis}. We let the reader check that the proof written there in the context of a Riemannian manifold only relies on Lemma \ref{lem:LM} and so holds in our more general setting.\\

By combining Theorem \ref{thm:rp} with Proposition \ref{prop:RHG}, we get the following result.

\begin{theorem} \label{thm:th}
Let $(X,\mu,\Gamma,L)$ be as in Section \ref{sec11} with $L$ self-adjoint and $p^L>\nu$. Assume that for some $q>2$ a reverse H\"older property $(RH_{q})$ is satisfied. Then there exists $\epsilon>0$ such that $(R_{p})$ holds for $p\in(2,2+\epsilon)$.
\end{theorem}

\begin{rem}
\begin{itemize}
\item According to Proposition \ref{prop:P2RH}, our assumptions are weaker than the Poincar\'e inequality $(P_2)$, so the previous result improves \cite[Theorem 0.4]{AC}.
\item We emphasise that the proof only relies on Gehring's result \cite{Gehring} (which can be thought of as an application of a Calder\'on-Zygmund decomposition), and never uses the very deep result of Keith-Zhong \cite{KZ} about the self-improvement of Poincar\'e inequalities which was used in \cite{AC}.
\end{itemize}
\end{rem}

\section{Boundedness of Riesz transforms for non self-adjoint operators without $(P_2)$} \label{sec:riesz}

Our aim in this section is to improve the understanding of the links between  gradient estimates $(G_p)$, boundedness of the Riesz transform $(R_p)$ through Poincar\'e inequalities, as initiated in \cite{ACDH} and \cite{AC}. Here, we focus on the case where the operator $L$ is not self-adjoint, but only sectorial.

\subsection{Caccioppoli inequality}

\begin{proposition}[$L^p$ Caccioppoli inequality] \label{prop:caccioppoli} 
Let $(X,\mu,\Gamma,L)$ be as in Section \ref{sec11}. Assume $(G_{p_0})$ for some $p_0\in(2,p^L)$. Suppose $p\in(2,p_0)$, $M \in \N$ and $N>0$. Then for every $f\in {\mathcal D}(L^M)$ and every ball $B$ of radius $r>0$,
$$ \left(\aver{B} |r \Gamma f|^p \,d\mu \right)^{1/p} \lesssim \sum_{\ell\geq 0}  2^{-\ell N} \left[\left(\aver{2^\ell B} |f|^p \,d\mu \right)^{1/p}+ \left(\aver{2^\ell B} |(r^2 L)^Mf|^p \, d\mu \right)^{1/p}\right].$$
Suppose in addition $M>\frac{\nu}{2}(\frac{1}{2}-\frac{1}{p})+\frac{1}{2}$. 
Then for every $f\in {\mathcal D}(L^M)$ and every ball $B$ of radius $r>0$,
$$ \left(\aver{B} |r \Gamma f|^p \,d\mu \right)^{1/p} \lesssim \sum_{\ell\geq 0}  2^{-\ell N} \left[\left(\aver{2^\ell B} |f|^2 \,d\mu \right)^{1/2}+ \left(\aver{2^\ell B} |(r^2 L)^Mf|^2 \, d\mu \right)^{1/2}\right].$$
\end{proposition}

\begin{rem} The second statement is more accurate than the first one. It is interesting to observe that it corresponds to a localised Sobolev embedding with the suitable scaling, since it yields a local version of (with the standard notation for Sobolev spaces) $W^{2M,2}_{L} \hookrightarrow W^{1,p}_{\Gamma}$ with the usual condition $(2M-1)>\nu\left(\frac{1}{2}-\frac{1}{p}\right)$.
\end{rem}

\begin{proof}
Let $p \in (2,p_0)$. For $f \in \mathcal{D}(L^M)$, write 
\begin{align*}
	f& =(I-(I-e^{-r^2L})^{2M})f + (I-e^{-r^2L})^{2M}f\\
	&=\sum_{k=1}^{2M}c_k e^{-kr^2L}f + (r^2L)^{-M}(I-e^{-r^2L})^{2M} (r^2L)^Mf,
\end{align*}
where $c_k$ are binomial coefficients. The statement of the proposition will then be a direct consequence of $L^p$-$L^p$ off-diagonal estimates ($L^2$-$L^p$ off-diagonal estimates, resp.) at scale $r$ of the operators $r\Gamma e^{-kr^2L}$ and $r\Gamma (r^2L)^{-M}(I-e^{-r^2L})^{2M}$. By interpolating $(G_{p_0})$ with \eqref{eq:DG}, one obtains $L^p$-$L^p$ off-diagonal estimates with exponential decay for the gradient of the semigroup, and in particular for given $\tilde{N}>0$
\begin{equation}
	\|r \Gamma e^{-r^2L} f \|_{L^p(E) \to L^p(F)} \lesssim \left(1+ \frac{d(E,F)}{r}\right)^{-\tilde{N}}
\label{eq:gradient}
\end{equation}
for all $r>0$ and all Borel sets $E,F \subset X$. On the other hand,  $e^{-r^2/2L}$ satisfies $L^2$-$L^p$ off-diagonal estimates by assumption. Composing these estimates with $L^p$-$L^p$ off-diagonal estimates for $r \Gamma e^{-r^2/2L}$  yields the desired $L^2$-$L^p$ off-diagonal estimates for $r \Gamma e^{-r^2L}$. 
For the second operator, one writes
\begin{align*}
	I-e^{-r^2L} = -\int_0^r \partial_s e^{-s^2L}\,ds
	=2(r^2L)S_r
\end{align*}
with
$$ S_r=\int_0^r \left(\frac{s}{r}\right)^2 e^{-s^2L} \,\frac{ds}{s},$$
therefore
\begin{equation}
	(r^2L)^{-M}(I-e^{-r^2L})^{2M} = cS_r^M \sum_{\ell=0}^M e^{-\ell r^2L}.
	\label{eq:Sr}
\end{equation}
From \eqref{eq:gradient}, we have that for $s \in (0,r)$, $s\Gamma e^{-s^2L}$ satisfies  $L^p$-$L^p$ off-diagonal estimates at scale $s$ and therefore also at scale $r$. Thus, 
\begin{align*}
	\|r\Gamma S_r\|_{L^p(E) \to L^p(F)} 
	&\lesssim \int_0^r \left(\frac{s}{r}\right)^2 \|r\Gamma e^{-s^2L} \|_{L^p(E) \to L^p(F)} \,\frac{ds}{s}\\
	& \lesssim \left(1+ \frac{d(E,F)}{r}\right)^{-\tilde{N}}.
\end{align*}
By composing  these off-diagonal estimates with $L^p$-$L^p$ off-diagonal estimates for $e^{-r^2L}$, one obtains from \eqref{eq:Sr} that the operator $r\Gamma (r^2L)^{-M}(I-e^{-r^2L})^{2M}$  satisfies $L^p$-$L^p$ off-diagonal estimates at scale $r$.
Since $\tilde{N}$ can be chosen arbitrarily, this already gives the first statement of the proposition. To finish the proof of the second statement, assume now that $M>\frac{\nu}{2}(\frac{1}{2}-\frac{1}{p})+\frac{1}{2}$. Let $B$ be a ball with radius $r$,  let $2\leq q_1\leq q_2 \leq p$. Then $L^{q_1}$-$L^{q_2}$ off-diagonal estimates for $s\Gamma e^{-s^2L}$ yield
\begin{align*}
	\left(\aver{B}|r\Gamma S_r f|^{q_2}\,d\mu\right)^{1/{q_2}}
	&\leq \int_0^r \left(\frac{s}{r}\right)^2 \left(\aver{B}|r\Gamma e^{-s^2L}f|^{q_2}\,d\mu\right)^{1/q_2} \,\frac{ds}{s}\\
	& \lesssim \sum_{\ell=0}^\infty 2^{-\ell N} \int_0^r \left(\frac{s}{r}\right)^{1-\nu(\frac{1}{q_1}-\frac{1}{q_2})} \,\frac{ds}{s} \left(\aver{2^\ell B}|f|^{q_1}\,d\mu\right)^{1/q_1}\\
	& \lesssim \sum_{\ell=0}^\infty 2^{-\ell N}  \left(\aver{2^\ell B}|f|^{q_1}\,d\mu\right)^{1/q_1},
\end{align*}
under the assumption that $\nu(\frac{1}{q_1}-\frac{1}{q_2})<1$.
An analogous computation yields
\begin{align*}
	\left(\aver{B}|S_r f|^{q_2}\,d\mu\right)^{1/{q_2}}
	& \lesssim \sum_{\ell=0}^\infty 2^{-\ell N} \int_0^r \left(\frac{s}{r}\right)^{2-\nu(\frac{1}{q_1}-\frac{1}{q_2})} \,\frac{ds}{s} \left(\aver{2^\ell B}|f|^{q_1}\,d\mu\right)^{1/q_1}\\
	& \lesssim \sum_{\ell=0}^\infty 2^{-\ell N}  \left(\aver{2^\ell B}|f|^{q_1}\,d\mu\right)^{1/q_1},
\end{align*}
under the assumption that $\nu(\frac{1}{q_1}-\frac{1}{q_2})<2$.
By iterating these off-diagonal estimates, we obtain that $r\Gamma S_r^M$ satisfies $L^2$-$L^p$ off-diagonal estimates. We finally combine these estimates with $L^2$-$L^2$ off-diagonal estimates for $e^{-r^2L}$, we then get $L^2$-$L^p$ off-diagonal estimates for the operator $r\Gamma (r^2L)^{-M}(I-e^{-r^2L})^{2M}$ from the representation \eqref{eq:Sr}.
\end{proof}

\subsection{Boundedness of Riesz transform under conservation property}

Let us first state an improvement of one the main results in \cite{ACDH}: the point is that we are able to replace $(P_2)$ by the weaker assumption $(P_{p_0})$ for $p_0>2$.

\begin{theorem}\label{thmbis} Let $(X,\mu,\Gamma,L)$ be as in Section \ref{sec11} with the conservation property for both $\Gamma$ and $L$.
Assume for some $p_0\in[2,p^L)$ that the combination $(P_{p_0})$ with $(G_{p_0})$ holds. If $p_0>2$ then $(R_p)$ holds for every $p\in(p_L ,p_0)$. If $p_0=2$ then there exists $\epsilon>0$ such that $(R_p)$ holds for every $p\in(p_L ,2+\epsilon)$.
\end{theorem}

\begin{rem} For $p_0<\nu$, where $\nu$ is the exponent in \eqref{dnu}, $(P_{p_0})$ is not necessary for $(R_p)$ to hold for every $p\in(1,p_0)$, as  the example of the connected sum of two copies of $\R^n$ shows (see {\rm \cite{CCH}}).
\end{rem}

\begin{rem} In \cite[Theorem 6.3]{BCF1}, it is shown that if $L$ is self-adjoint then $(G_p)$ with $(P_p)$ for some $p>2$ implies $(P_2)$. This argument goes trough a (perfectly) localised Caccioppoli inequality. In the case where the operator is not self-adjoint, such an inequality is unknown. Even if we are showing a weaker version with off-diagonal terms (Proposition \ref{prop:caccioppoli}), it is not clear how this would allow us to adapt the arguments of \cite{BCF1} in order to obtain $(P_2)$ from $(G_p)+(P_p)$. 

So as far as we know, Theorem \ref{thmbis} is a real improvement with respect to \cite{ACDH}.
\end{rem}

\begin{proof}[Proof of Theorem $\ref{thmbis}$] We only consider the case $p_0>2$. The case $p_0=2$ can be treated similarly.
We apply Proposition \ref{prop:extrapolation} to the Riesz transform ${\mathcal R}:=\Gamma L^{-1/2}$. Let $q\in(2,p_0)$, and let $D \in \N$ with $D>\frac{\nu}{4}$. Moreover, let $f \in L^2(X,\mu)$ and $B$ be a ball of radius $r>0$. For $t>0$, denote 
$$ P_t^{(D)}:=I - (I-e^{-tL})^D.$$ 

We check in two steps the assumptions of Proposition \ref{prop:extrapolation} for  $T={\mathcal R}$ and $A_{r^2}= P_{r^2}^{(D)}$.

\medskip
\noindent {\bf Step 1:} Checking of \eqref{eq:ass1}.

Following \cite[Lemma 3.1]{ACDH}, which only relies on the Davies-Gaffney estimates \eqref{eq:DG}, we already know that \eqref{eq:ass1} is satisfied for  $T={\mathcal R}$ and $A_{r^2}=P_{r^2}^{(D)}$. More precisely, it is proven that 
\begin{equation} \label{eq:ass1-bis}
 \left( \aver{B} |\Gamma L^{-1/2}  (I-P^{(D)}_{r^2} )f |^{2} \, d\mu \right)^{1/2} \lesssim  \sum_{j=0}^\infty  2^{j(\frac{\nu}{2}-2D)} \left(\aver{2^j B} |f|^2\, d\mu \right)^{1/2},
\end{equation}
which in particular yields \eqref{eq:ass1}.

\medskip
\noindent {\bf Step 2:} Checking of \eqref{eq:ass2}.

Let us now check \eqref{eq:ass2}, again with $T={\mathcal R}$ and $A_{r^2}=P_{r^2}^{(D)}$, which is 
\begin{equation} \label{eq:ass2-bis} 
\left( \aver{B} |\Gamma L^{-1/2} P_{r^2}^{(D)}f |^{q} d\mu \right)^{\frac{1}{q}} \lesssim  \inf_{x\in B} \ \left[\calM_2( \Gamma L^{-1/2} f)(x) + \calM_2 (f)(x) \right].
\end{equation}
By Lemma \ref{lem} and Remark \ref{rem1}, applied to $g=L^{-1/2} f$ and $M \in \N$ with $M>\frac{1}{2}(\frac{\nu}{p}+1)$, we deduce that
\begin{align} \label{eq:.}
 \left( \aver{B} |\Gamma L^{-1/2}  P_{r^2}^{(D)}f |^{q} d\mu \right)^{\frac{1}{q}} \lesssim \left( \aver{2 B} |\Gamma L^{-1/2}    P_{r^2}^{(D)}f |^{2} d\mu \right)^{\frac{1}{2}} + \sum_{\ell \geq 0} 2^{-\ell N} \left( \aver{2^\ell B} | (r^2L)^{M-1/2} e^{- r^2 L} f |^{p_0} d\mu \right)^{\frac{1}{p_0}}.
\end{align} 
Let us now estimate the two quantities. We get from \eqref{eq:ass1-bis}
\begin{align} \nonumber
& \left(\aver{2B} | \Gamma L^{-1/2} P_{r^2}^{(D)} f |^2 \, d\mu \right)^{1/2} \\ \nonumber
& \qquad  \leq \left(\aver{2 B} | \Gamma L^{-1/2} f |^2 \, d\mu \right)^{1/2} 
+ \left(\aver{2B} | \Gamma L^{-1/2}  (I-P_{r^2}^{(D)}) f |^2 \, d\mu \right)^{1/2}  \\ \nonumber
& \qquad \lesssim \left(\aver{2 B} | \Gamma L^{-1/2} f |^2 \, d\mu \right)^{1/2}  + \inf_{x\in B} \calM_2 (f)(x) \\
& \qquad \lesssim  \inf_{x\in B} \ \left[\calM_2 (\Gamma L^{-1/2} f)(x) + \calM_2 (f)(x) \right].
\label{eq:ass2-P}
\end{align}
Moreover, by covering $2^{\ell}B$ with approximately $2^{\nu \ell}$ balls of radius $r$, we obtain from Lemma \ref{lemma} 
\begin{align*}
\left( \aver{2^\ell B} | (r^2L)^{M-1/2} e^{- \frac{r^2}{2} L} f |^{p_0} d\mu \right)^{\frac{1}{p_0}} & \lesssim 2^{2\nu\ell}  \sup_{Q\supset B} \left(\aver{Q} |f|^2 \, d\mu \right)^{1/2} \\
& \lesssim 2^{2\nu \ell} \inf_{x\in B} \ \calM_2 (f)(x).
\end{align*}
Consequently,
plugging these two last estimates into \eqref{eq:.} with $N$ chosen large enough, we deduce 
\begin{align*}  
\left( \aver{B} |\Gamma L^{-1/2} P_{r^2}^{(D)} f |^{q} d\mu \right)^{1/q} \lesssim \inf_{x\in B} \  \left[\calM_2 (\Gamma L^{-1/2} f)(x) + \calM_2 (f)(x)\right],
\end{align*}
which is \eqref{eq:ass2-bis}. In this way, we obtain \eqref{eq:ass2}, and the proof is complete by Proposition \ref{prop:extrapolation}. To be more precise, Proposition \ref{prop:extrapolation} implies that the Riesz transform is bounded on $L^p(X,\mu)$, for every $p\in(2,q)$ with arbitrary $q\in(2,p_0)$.
\end{proof}

\begin{lemma} \label{lem} Assume the conservation property for $L$ and $\Gamma$ as well as $(G_{p_0})$ and $(P_{p_0})$ for some $p_0 \in [2,p^L)$. Suppose $N>0$, and $M \in \N$ with $M>\frac{1}{2}(\frac{\nu}{p_0}+1)$. If $p_0>2$, suppose $p\in[2,p_0)$. Then for every $g \in \mathcal{D}(L^M)$, every $t>0$ and every ball $B$ of radius $\sqrt{t}$,
$$ \left(\aver{B} |\sqrt{t} \Gamma e^{-tL} g|^{p} \, d\mu \right)^{1/p} \lesssim \left(\aver{2B} |\sqrt{t} \Gamma e^{-tL} g|^{2} \, d\mu \right)^{1/2} + \sum_{\ell \geq 0} 2^{-\ell N} \left(\aver{2^\ell B} | (tL)^M e^{-tL} g |^{p_0}  \, d\mu \right)^{1/p_0}.$$
If $p_0=2$, then there exists $\epsilon>0$ such that the same property holds for every $p\in(2,2+\epsilon)$.
\end{lemma}

\begin{rem} \label{rem1} As we will see in the proof, the operators $\sqrt{t} \Gamma e^{-tL}$ in the LHS and RHS can be replaced by $\sqrt{t} \Gamma P_t^{(D)}$ for any integer $D$.
 \end{rem}

\begin{proof} We only consider the case $p_0>2$. The case $p_0=2$ can be treated similarly with $p_0=p=2$, since we already have Davies-Gaffney estimates \eqref{eq:DG} by assumption.
Let $g \in \mathcal{D}(L^M)$, let $B$ be a ball of radius $\sqrt{t}$. 
From Proposition \ref{prop:caccioppoli} with the conservation property, we then have for every $r \in (0,\sqrt{r})$ and every ball $B_r \subset B$
\begin{align*}
\left(\aver{B_r} |\Gamma e^{-tL} g|^{p} \,  d\mu \right)^{1/p} &\lesssim \sum_{\ell \geq 0} 2^{-\ell \tilde N} \left[r^{-1} \left(\aver{2^\ell B_r} \left|e^{-tL} g - \left(\aver{B_r} e^{-tL} g d\mu \, \right) \right|^{p} \,  d\mu \right)^{1/p} \right. \\
& \qquad \qquad \qquad +  \left. r^{-1} \left( \aver{2^\ell B_r} |(r^2 L)^Me^{-tL}g|^{p} \,  d\mu \right)^{1/p}\right],
\end{align*}
where $M \in \N$, and $\tilde{N}>0$ can be chosen arbitrarily large. 
Using $(P_{p_0})$, which self-improves into $(P_{p_0-\kappa})$ for $\kappa$ small enough (see \cite{KZ}), and Lemma \ref{met}, we deduce that for $\kappa>0$ small enough and $p\in(p_0-\kappa,p_0)$, we have
\begin{align*}
\frac{1}{r} \left(\aver{2^\ell B_r} \left|e^{-tL} g - \left(\aver{B_r} e^{-tL} g\,d\mu \right) \right|^{p} \,  d\mu \right)^{1/p} & \lesssim \sum_{j=0}^{\ell} \frac{1}{r} \left(\aver{2^j B_r} \left|e^{-tL} g - \left(\aver{2^jB_r} e^{-tL} g \,d\mu \right) \right|^{p} \,  d\mu \right)^{1/p} \\
& \lesssim \sum_{j=0}^{\ell} 2^j \left(\aver{2^j B_r} |\Gamma e^{-tL} g|^{p_0-\kappa} \, d\mu \right)^{1/p_0-\kappa}.
\end{align*}
We thus obtain, up to changing $\tilde N$ to some other parameter $N$ (but which can still be chosen arbitrarily large), for every $r\in(0,\sqrt{t})$ and every ball $B_r \subset B$ 
\begin{align}
\label{gg}\left(\aver{B_r} |\Gamma e^{-tL} g|^{p} \,  d\mu \right)^{1/p} 
\lesssim &\sum_{\ell \geq 0} 2^{-\ell  N} \left[\left(\aver{2^\ell B_r} |\Gamma e^{-tL} g|^{p_0-\kappa} \,  d\mu \right)^{1/(p_0-\kappa)} \right.\\ \nonumber
& \qquad \qquad  +  \left. r^{-1} \left( \aver{2^\ell B_r} |(r^2 L)^Me^{-tL}g|^{p_0} \,  d\mu \right)^{1/p_0}\right].
\end{align}
Now consider the function
$$ h := |\Gamma e^{-tL} g| + c(g,B),$$
where 
$$ c(g,B):=t^{-1/2}\sum_{\ell \geq 0} 2^{-\ell  N} \left(\aver{2^\ell B} |(tL)^M e^{-tL} g|^{p_0} \, d\mu \right)^ {1/p_0}$$
is a constant. 
 We have by using \eqref{dnu} and $r\leq \sqrt{t}$
\begin{align*}
r^{-1} \left( \aver{2^\ell B_r} |(r^2L)^Me^{-tL}g|^{p_0} \,  d\mu \right)^{1/p_0} 
 & \lesssim \left(\frac{r}{\sqrt{t}}\right)^{2M-1}  \left(\frac{\sqrt{t}}{r}\right)^{\nu/p_0} t^{-1/2}
  \left( \aver{2^\ell B} |(tL)^M e^{-tL}g|^{p_0} \,  d\mu \right)^{1/p_0} \\
  &\lesssim  t^{-1/2}
  \left( \aver{2^\ell B} |(tL)^M e^{-tL}g|^{p_0} \,  d\mu \right)^{1/p_0},
 \end{align*}
 choosing $M>\frac{1}{2}(\frac{\nu}{p_0}+1)$.
Therefore  \eqref{gg} yields 
\begin{eqnarray*}
 \left(\aver{B_r} |\Gamma e^{-tL} g|^{p} \,d\mu \right)^{1/p}  &\lesssim& \sum_{\ell \geq 0} 2^{-\ell N} \left(\aver{2^\ell B_r} |\Gamma e^{-tL} g|^{p_0-\kappa} \,  d\mu \right)^{1/(p_0-\kappa)}    +  c(g,B).
\end{eqnarray*}
Since $c(g,B)$ is a constant (and it is also equal to any of its average), it follows that
\begin{equation} \label{eq:ouf}
\left(\aver{B_r} h^{p} \,d\mu \right)^{1/p} \lesssim \sum_{\ell \geq 0} 2^{-\ell N} \left(\aver{2^\ell B_r} h^{p_0-\kappa} \,  d\mu \right)^{1/(p_0-\kappa)}.  
\end{equation}
Now for $x\in 2B$, we consider the quantities
$$ F_{p}(x):=\sup_{x\in B_r\subset 2 B} \left(\aver{B_r} h^{p} \, d\mu \right)^{1/p}$$
where the supremum is taken over all the balls $B_r$ containing $x$ and included into $2B$ and
$$ F_{p_0-\kappa}(x):=\sup_{x\in B_r } \left(\aver{B_r} h^{p_0-\kappa} \, d\mu \right)^{1/(p_0-\kappa)}:={\mathcal M}_{p_0-\kappa}[h](x),$$
We have proved that for every $x\in 2B$
$$ F_{p}(x) \lesssim F_{p_0-\kappa}(x).$$
We may then apply Gehring's Lemma (see Theorem \ref{thm:gehring}), and we deduce that for some $\eps>0$
$$ \left(\aver{B} |h|^{p+\epsilon} d\mu \right)^{1/(p+\epsilon)} \lesssim \left(\aver{2B} |h|^{p} d\mu \right)^{1/p}.$$
By the self-improvement of reverse H\"older inequality (see \cite[Appendix B]{BCF1}), we deduce that
$$ \left(\aver{B} |h|^{p+\epsilon} d\mu \right)^{1/p} \lesssim \left(\aver{2B} |h|^{2} d\mu \right)^{1/2}.$$
In particular due to the definition of $h$, we have 
\begin{align*}
 \left(\aver{B} |\Gamma e^{-tL}g|^{p+\epsilon} d\mu \right)^{1/p} & \lesssim \left(\aver{2 B} |\Gamma e^{-tL}g|^{2} d\mu \right)^{1/2}+c(g,B),
  \end{align*}
which is the desired inequality.
\end{proof}

\section{Boundedness of Riesz transforms under elliptic perturbation} \label{sec:appli}

Consider $(M,d,\mu)$ a doubling Riemannian manifold (satisfying \eqref{dnu}), equipped with the Riemannian gradient $\nabla$ and its divergence operator $\dive = \nabla^*$, such that the self-adjoint Laplacian is given by $\Delta:=\dive \nabla$ and for every $f,g \in {\mathcal D}(\Delta)$ we have
$$ \int_M f \Delta(g) \, d\mu = \int_M \langle \nabla f(x), \nabla g(x) \rangle_{T_xM} d\mu(x),$$
where $T_xM$ is the tangent space of $M$ at $x$.

Let $A = A(x)$ be a complex matrix - valued function, defined on $M$ and satisfying the ellipticity (or accretivity) condition
\begin{equation} \lambda |\xi|^ 2 \leq  \Re \langle A(x) \xi, \xi \rangle \qquad \textrm{and} \qquad  |\langle A(x)\xi , \zeta\rangle | \leq \Lambda |\xi||\zeta|, \label{eq:ell} \end{equation}
for some constants $\lambda,\Lambda>0$ and every $x\in M$, $\xi,\zeta\in T_xM$. 

To such a complex matrix valued function $A$, we may define a second order divergence form operator
$$ L=L_A f :=- \dive (A\nabla f),$$
which we first interpret in the sense of maximal accretive operators via a
ses-quilinear form. That is, ${\mathcal D}(L)$ is the largest subspace contained in $W^{1,2}:={\mathcal D}(\nabla)$
for which 
$$ \left| \int_M \langle A \nabla f, \nabla g\rangle \, d\mu \right| \leq C \|g\|_2 \qquad \forall g \in W^{1,2},$$ 
and we define $Lf$ by
$$ \langle Lf, g \rangle = \int_M \langle A \nabla f, \nabla g\rangle \, d\mu$$
for $f \in {\mathcal D}(L)$ and $g \in W^{1,2}$. Thus defined, $L=L_A$ is a maximal-accretive operator
on $L^2$ and ${\mathcal D}(L)$ is dense in $W^{1,2}$.

Let us motivate such considerations. Consider on the manifold $(M,d,\mu)$ two complete Riemannian metrics $G_o = \langle, \rangle_o$ and $G = \langle,\rangle$.  Assume that these two metrics are quasi-isometric in the sense that the associated
lengths $|\ |$ and $|\ |_0$ on tangent vectors satisfy
$$ c_1 |v|_{T_xM} \leq  |v|_{0,T_xM} \leq c_2 |v|_{T_xM}$$
for some numerical constants $c_1,c_2$ and every $x \in M$ and tangent vector $v\in T_xM$. This implies (see \cite{Barbatis} and \cite[Section 4]{CoDu}) that there exists an automorphism $A = (A(x))_{x\in M}$ of the tangent bundle $TM$ such that for every $x\in M$ and $u,v\in T_xM$
$$ \langle A(x)u, v \rangle = \langle u, v\rangle_0.$$
Moreover, the matrix-valued map $A$ satisfies the ellipticity condition \eqref{eq:ell}.
So there is an equivalence between the two points of view: to perturb the Laplace operator through the map $A$ and to perturb the considered Riemannian metric on the manifold.

\bigskip

We denote by $RH_{q}(\Delta)$ and $RH_{q}(L)=RH_{q}(L_A)$ the reverse H\"older inequality property for harmonic functions associated respectively with the Laplace operator $\Delta$, or with the perturbed operator $L=L_A$.
For $A$ and $p\in(2,\infty)$, we set $(G_{p,A})$ and $(R_{p,A})$ for the $L^p$-boundedness property of the gradient of the semigroup or the Riesz transforms associated with the operator $L=L_A$.

Let us also point out that in such a context it is well-known that (even if $A$ is non self-adjoint) Lemma \ref{lem:LM} holds.

\begin{proposition} \label{prop:RHbis} Let $(M,d,\mu)$ be a doubling Riemannian manifold with \eqref{FKR}.
Let us assume $RH_{q}(\Delta)$ for some $q>2$. Then there exists $\epsilon_0>0$ such that for every $p\in(2,2+\epsilon_0)$ and every $A$ with $ \| A-\textrm{Id}\|_\infty \leq \epsilon_0$, Property $RH_{p}(L_A)$ holds.
\end{proposition}

\begin{proof}
Consider a ball $B_0$ and a function $f\in {\mathcal D}(L_A)$, $L_A$-harmonic on $2B_0$. Then using Lemma \ref{lem:LM} for every ball $B\subset B_0$, there exists $u\in W^{1,2}$ such that $f-u\in W^{1,2}$ is supported in the ball $2B$ and $u$ is $\Delta$-harmonic in $2B$. Moreover, by repeating the argument employed for Lemma \ref{lem:LM}, we have
\begin{align*}
\aver{B} |\nabla(f-u)|^2 \, d\mu & = \frac{1}{|B|}\int \langle  \nabla (f-u), \nabla (f-u) \rangle \, d\mu \\
 & = \frac{1}{|B|}\int_M \langle  \nabla (f-u), \nabla f \rangle \, d\mu,
\end{align*}
where we have used the $\Delta$-harmonicity of $u$ on $2B$ (containing the support of $f-u$).
Then using the $L_A$-harmonicity of $f$ on $2B$, we deduce
\begin{align*}
\aver{B} |\nabla(f-u)|^2 \, d\mu & = \frac{1}{|B|}\int_M \langle (\textrm{Id}-A) \nabla (f-u), \nabla f \rangle \, d\mu \\
 & \leq \|A-\textrm{Id}\|_\infty \left(\aver{B} |\nabla(f-u)|^2 \, d\mu\right)^{1/2} \left( \aver{B} |\nabla f|^2 \, d\mu \right)^{1/2}.
\end{align*}
Hence,
\begin{align*}
\left(\aver{B} |\nabla(f-u)|^2 \, d\mu\right)^{1/2} & \leq \|A-\textrm{Id}\|_\infty \left( \aver{B} |\nabla f|^2 \, d\mu \right)^{1/2}.
\end{align*}
Then
\begin{align*}
 \left(\aver{B} |\nabla f |^{2} \, d\mu \right)^{1/2} & \leq \left(\aver{B} |\nabla u |^{2} \, d\mu \right)^{1/2} + \left(\aver{B} |\nabla (f-u) |^{2} \, d\mu \right)^{1/2} \\
 & \leq \left(\aver{B} |\nabla u |^q \, d\mu \right)^{1/q} + C_\nu \left(\aver{2B} |\nabla (f-u) |^{2} \, d\mu \right)^{1/2} \\
 & \leq C_{0} \left(\aver{2B} |\nabla u | \, d\mu \right) + C_\nu \|A-\textrm{Id}\|_\infty \left( \aver{2B} |\nabla f|^2 \, d\mu \right)^{1/2},
\end{align*}
where $C_0$ is the implicit constant in $RH_{q}(\Delta)$ (which self-improves into a $L^1-L^q$ reverse H\"older inequality, see Remark \ref{rem:rh}) and $C_\nu$ is a constant only depending on the doubling condition \eqref{dnu}.
Consequently, 
\begin{align*}
 \left(\aver{B} |\nabla f |^{2} \, d\mu \right)^{1/2} & \leq C_{0} \left(\aver{2B} |\nabla f | \, d\mu \right) + C_{0} \left(\aver{2B} |\nabla (f-u) |^2 \, d\mu \right)^{1/2} \\
 & \qquad + C_\nu \|A-\textrm{Id}\|_\infty \left( \aver{2B} |\nabla f|^2 \, d\mu \right)^{1/2} \\
  & \leq C_{0} \left(\aver{2B} |\nabla f | \, d\mu \right) + (C_{0} + C_\nu) \|A-\textrm{Id}\|_\infty \left( \aver{2B} |\nabla f|^2 \, d\mu \right)^{1/2}.
\end{align*}
Then consider $\epsilon_1=\eta (C_{0} + C_\nu)^{-1}$ (for a small parameter $\eta$, to be fixed later) and assume that $\|A-\textrm{Id}\|_\infty \leq \epsilon_1$. We obtain
\begin{align}
 \left(\aver{B} |\nabla f |^{2} \, d\mu \right)^{1/2} & \leq C_{0} \left(\aver{2B} |\nabla f | \, d\mu \right) + \eta \left( \aver{2B} |\nabla f|^2 \, d\mu \right)^{1/2}. \label{eq:eqboule}
\end{align}
The previous inequality holds for every balls $B\subset B_0$. 

Then let us fix a ball $B_1 \subset B_0$ and consider $(Q^j)_j$ a Whitney covering of $B_1$, which is a collection of balls satisfying in particular
 $$ \textrm{$(Q^j)_j$ is a covering of $B_1$ and $(2Q^j)_j$ is also a bounded covering of $B_1$.} $$
So we have
$$  \aver{B_1} |\nabla f |^{2} \, d\mu \leq \frac{1}{|B_1|} \sum_j |Q^j| \left(\aver{Q^j} |\nabla f|^2 \, d\mu\right)$$
and by applying \eqref{eq:eqboule} to each $Q^j$, we get
$$  \aver{B_1} |\nabla f |^{2} \, d\mu \leq 2 C_0^2 \frac{1}{|B_1|} \sum_j |Q^j| \left(\aver{2Q^j} |\nabla f | \, d\mu \right)^2 + \frac{2\eta^2}{|B_1|} \sum_j |Q^j| \left(\aver{2Q^j} |\nabla f|^2 \, d\mu\right).$$
Since $(2Q^j)_j$ is also a bounded covering of $B_1$, this implies, with $C_1,C_2$ numerical constants,
$$  \aver{B_1} |\nabla f |^{2} \, d\mu \leq C_1 \sup_{B\subset B_1} \left(\aver{B} |\nabla f | \, d\mu \right)^2 + C_2 \eta^2 \left(\aver{B_1} |\nabla f|^2 \, d\mu\right).$$
We then choose $\eta$ small enough such that $\eta^2<\frac{1}{2C_2}$. Hence
$$  \aver{B_1} |\nabla f |^{2} \, d\mu \leq C_1 \sup_{B\subset B_1} \left(\aver{B} |\nabla f | \, d\mu \right)^2 + \frac{1}{2} \left(\aver{B_1} |\nabla f|^2 \, d\mu\right),$$
which yields
$$  \left(\aver{B_1} |\nabla f |^{2} \, d\mu\right)^{1/2}  \leq \sqrt{2C_1} \sup_{B\subset B_1} \left(\aver{B} |\nabla f | \, d\mu \right).$$
This holds uniformly for every ball $B_1\subset B_0$, in particular, we have the following maximal inequality
$$ \sup_{B_1 \subset B_0} \left(\aver{B_1} |\nabla f |^{2} \, d\mu\right)^{1/2}  \leq \sqrt{2C_1} \sup_{B\subset B_1} \left(\aver{B} |\nabla f | \, d\mu \right).$$
We may then apply Gehring's result \cite{Gehring} (Theorem \ref{thm:gehring}), which yields that there exits $\epsilon_2$ such that
$$  \left(\aver{B_0} |\nabla f |^{q+\epsilon_2} \, d\mu\right)^{1/(2+\epsilon_2)}  \lesssim \left(\aver{2B_0} |\nabla f |^{2} \, d\mu\right)^{1/2},$$
with an implicit constant which is independent of $f$ and $B_0$. The proof is complete by choosing $\epsilon_0:=\min(\epsilon_1,\epsilon_2)$.
\end{proof}

\begin{proposition} \label{prop:RHGbis} Let $(M,d,\mu)$ be a doubling Riemannian manifold with \eqref{FKR} and $RH_{q}(\Delta)$ for some $q>2$. Then there exists $\epsilon_1>0$ such that for every $p\in(2,2+\epsilon_1)$ and every $A$ with $ \| A-\textrm{Id}\|_\infty \leq \epsilon_1$, Property $(G_{p,A})$ holds whenever $p^L>\nu$.
\end{proposition}

\begin{proof} Let $t>0$ and let $B_0$ be a ball of radius $\sqrt{t}$. Denote $L:=L_A$. Then for every $f\in L^2$ and every sub-ball $B\subset B_0$ of radius $r\leq \sqrt{t}$, we have
$$ \left( \aver{B} | \nabla e^{-tL} f |^2 \, d\mu \right)^{1/2} \leq \left( \aver{B} | \nabla u |^2 \, d\mu \right)^{1/2} + \left( \aver{B} | \nabla (u-e^{-tL} f) |^2 \, d\mu \right)^{1/2}$$
where (according to Lemma \ref{lem:LM}) $u$ is a function $L$-harmonic on $2B$ such that $e^{-tL}f -u$ is supported on $2B$.
According to the previous proposition, we know that if $\epsilon_0>0$ is small enough and $ \| A-\textrm{Id}\|_\infty \leq \epsilon_0$, then $RH_{q}(\Delta)$ implies $RH_{\tilde q}(L_A)$ (for some $\tilde q\in(2,2+\epsilon_0)$) and so by Remark \ref{rem:rh} 
\begin{align*} \left( \aver{B} | \nabla u |^2 \, d\mu \right)^{1/2} & \lesssim \left( \aver{2B} | \nabla u | \, d\mu \right) \\
& \lesssim \left( \aver{2B} | \nabla e^{-tL} f | \, d\mu \right) + \left( \aver{2B} | \nabla (e^{-tL}f -u) |^2 \, d\mu \right)^{1/2}.
\end{align*}
Hence, 
$$ \left( \aver{B} | \nabla e^{-tL} f |^2 \, d\mu \right)^{1/2} \lesssim \left( \aver{2B} | \nabla e^{-tL} f | \, d\mu \right) + \left( \aver{2B} | \nabla (e^{-tL}f -u) |^2 \, d\mu \right)^{1/2}.$$
For the second term, we get, using the $L$-harmonicity of $u$ with the support of $e^{-tL}f-u$ and the accretivity of the matrix-valued map $A$,
\begin{align*}
\int_{2B} |  \nabla (e^{-tL}f -u) |^2 \, d\mu & \lesssim \int_{2B} |  A \nabla (e^{-tL}f -u) |^2 \, d\mu = \int_{2B} (e^{-tL}f -u) L (e^{-tL}f -u) \, d\mu \\
 & \lesssim \int_{2B} (e^{-tL}f -u) Le^{-tL}f \, d\mu \\
 & \lesssim \left(\int_{2B} |e^{-tL}f -u|^2 \, d\mu\right)^{1/2} \left(\int_{2B} |Le^{-tL} f|^2 \, d\mu \right)^{1/2} \\
 & \lesssim \left(\int_{2B} |\nabla (e^{-tL}f -u)|^2 \, d\mu\right)^{1/2} r \left(\int_{2B} |Le^{-tL} f|^2 \, d\mu \right)^{1/2},
\end{align*}
where we used the Faber-Krahn inequality \eqref{FKR} in the last step.
We deduce 
$$ \left( \aver{2B} | \nabla (e^{-tL}f -u) |^2 \, d\mu \right)^{1/2} \lesssim \left(\aver{2B} |\sqrt{t} Le^{-tL} f|^2 \, d\mu \right)^{1/2}.$$
Since $p^L>\nu$ and $r\leq \sqrt{t}$, the doubling property \eqref{dnu} yields for some $q\in(\nu,p^L)$
\begin{align*} 
r \left(\aver{2B} |\sqrt{t} Le^{-tL} f|^2 \, d\mu \right)^{1/2} & \lesssim \left(\frac{r}{\sqrt{t}}\right) \left(\aver{2B} |\sqrt{t} Le^{-tL} f|^{q} \, d\mu \right)^{1/q} \\
 & \lesssim \left(\frac{r}{\sqrt{t}}\right) \left(\frac{\sqrt{t}}{r}\right)^{\nu/q} \left(\aver{2B_0} |\sqrt{t} Le^{-tL} f|^{q} \, d\mu \right)^{1/q} \\
 & \lesssim \left(\aver{2B_0} |\sqrt{t} Le^{-tL} f|^{q} \, d\mu \right)^{1/q}.
\end{align*}
Consequently for every $p\in(p_L,2)$ (due to $L^p$-$L^q$ off-diagonal estimates of $tLe^{-tL}$ at the scale $\sqrt{t}$ by analyticity), we have
\begin{align*}
 \left( \aver{B} | \sqrt{t} \nabla e^{-tL} f |^2 \, d\mu \right)^{1/2} \lesssim \left( \aver{B} | \sqrt{t} \nabla e^{-tL} f | \, d\mu \right) + \inf_{x\in B_0} {\mathcal M}_p[f](x).  
\end{align*}
The last term is a constant and independent of the generic ball $B \subset B_0$.
We may now apply Gehring's result (see Theorem \ref{thm:gehring}), which yields that there exists $\eta>0$ (independent on $t>0$, $B_0$ and $f$) such that 
\begin{align*}
 \left( \aver{B_0} | \sqrt{t} \nabla e^{-tL} f |^{2+\eta} \, d\mu \right)^{1/(2+\eta)} \lesssim \left( \aver{2B_0} | \sqrt{t} \nabla e^{-tL} f |^2 \, d\mu \right)^{1/2} + \left( \aver{2B_0} | {\mathcal M}_p[f] |^2 \, d\mu \right)^{1/2}.  
\end{align*}
Using the $L^2$ Davies-Gaffney estimates for $L=L_A$ (see for example \cite[Proposition 2.1]{memoirs} for the classical perturbation argument), we have
\begin{align*}
 \left( \aver{B_0} | \sqrt{t} \nabla e^{-tL} f |^{2+\eta} \, d\mu \right)^{1/(2+\eta)} \lesssim \left( \aver{2B_0}  {\mathcal M}_2[f]^{2+\eta} \, d\mu \right)^{1/(2+\eta)}.  
\end{align*}
The implicit constant in this last inequality as well as $\eta$ are independent of the initial ball $B_0$ of radius $\sqrt{t}$.
Then by covering the whole ambiant space with a collection of balls of radius $\sqrt{t}$ such that the double balls have a bounded overlap, we deduce that
\begin{align*}
 \| \sqrt{t} \nabla e^{-tL} f \|_{2+\eta} \lesssim  \|f\|_{2+\eta},  
\end{align*}
where $\eta$ and the implicit constant are independent of $t>0$.
That corresponds to $(G_{2+\eta,A})$. We conclude the proof by choosing $\epsilon_1:=\min(\eta,\epsilon_0)$.
\end{proof}

In order to check the condition $p^L>\nu$, we may use characterisations in terms of Sobolev inequalities.
For $q>2$, the scale-invariant local Sobolev inequality $(LS_q)$ holds if
\begin{equation*}\label{Sq}
\tag{$LS_q$} \|f \|_{q}^2\lesssim \frac{1}{|B|^{1-\frac{2}{q}}}\left(\|f\|^2_2+r^2 \||\nabla f| \|_2^2\right),
\end{equation*}
for every ball $B$ of radius $r>0$,  every $f\in\mathcal{F}$ supported in $B$.
This inequality was introduced in \cite{S} and was shown, under \eqref{d}, to be equivalent  to the upper Gaussian estimates for the heat kernel \eqref{due} and \eqref{UE}, in the Riemannian setting (see also \cite{ST2} and \cite{BCS} for many reformulations of  \eqref{Sq} and an alternative proof of the equivalence with \eqref{due}).

We deduce the following result.

\begin{theorem}  Let $(M,d,\mu)$ be a doubling Riemannian manifold with \eqref{FKR}. Let us assume $RH_{q}(\Delta)$ for some $q>2$. Assume also that \eqref{due} holds for the Laplacian. Then there exists $\varepsilon>0$ such that for every $p\in(2,2+\varepsilon)$ and every map $A$ with $ \| A-\textrm{Id}\|_\infty \leq \varepsilon$ and $L_A$ self-adjoint, the properties $RH_{p}(L_A)$ and $(R_{p,A})$ are satisfied.
\end{theorem}

\begin{proof} Since $A$ satisfies an accretivity condition, it is known that \eqref{Sq} for $-\Delta$ implies \eqref{Sq} for $L_A$. Following the characterisation of upper estimates through Sobolev inequalities, we deduce that \eqref{due} for $-\Delta$ implies \eqref{due} for $L_A$ and so $p^L=\infty$. We then conclude by combining Theorem \ref{thm:rp} with Propositions \ref{prop:RHbis} and \ref{prop:RHGbis}. 
\end{proof}

\end{document}